\documentclass{article}
\usepackage[utf8]{inputenc}
\usepackage[utf8]{inputenc}
\usepackage{amsmath}
\usepackage[left=3.4cm, right=3.4cm, bottom=3cm, top=3cm]{geometry}
\usepackage{hyperref}
\usepackage{cleveref}
\usepackage{amssymb}
\usepackage{mathtools}
\usepackage{tikz-cd}
\usepackage{amsthm}
\usepackage{graphicx}
\usepackage{graphics}
\usepackage{mathrsfs}
\usepackage{cleveref}
\usepackage{thmtools}
\usepackage{enumitem}
\usepackage{bbm}
\newlist{steps}{enumerate}{1}
\usepackage{array}
\setlist[steps, 1]{label = Step \arabic*:}
\theoremstyle{plain}
\usepackage{footnote}
\makesavenoteenv{tabular}
\newtheorem*{theorem*}{Theorem}
\newtheorem{theorem}{Theorem}[subsection]
\newtheorem{definition}[theorem]{Definition}

\newtheorem{proposition}[theorem]{Proposition}
\newtheorem{remark}[theorem]{Remark}
\newtheorem{corollary}[theorem]{Corollary}
\newtheorem{exmp}[theorem]{Example}

\newtheorem{lemma}[theorem]{Lemma}
\newtheorem{conjecture}[theorem]{Conjecture}
\numberwithin{equation}{subsection}
\usepackage{rotating}
\usepackage{comment}
\usepackage[nottoc,notlot,notlof]{tocbibind} 

\usepackage[toc,page]{appendix}

\title{A Groupoid Construction of Functional Integrals: Brownian Motion and Some TQFTs}
\author{Joshua Lackman\footnote{josh@pku.edu.cn} }
\date{}
\begin{document}

\maketitle
\begin{abstract}
\noindent We formalize Feynman's construction of the quantum mechanical path integral. To do this, we shift the emphasis in differential geometry from the tangent bundle onto the pair groupoid. This allows us to use the van Est map and the piecewise linear structure of manifolds to develop a coordinate-free, partition of unity-free approach to integration of differential forms, etc. This framework makes sense for any sigma model valued in a Lie algebroid. We apply it to define the Wiener measure, stochastic integrals and other observables in a coordinate-free way. We use it to reconstruct Chern-Simons with finite gauge group and to obtain some non-perturbative deformation quantizations via the Poisson sigma model on a disk.
\end{abstract}
\tableofcontents
\section{Introduction and Overview}
\subsection{Feynman's Construction of Path Integrals}
Given a Lagrangian $\mathcal{L}(\mathbf{x},\dot{\mathbf{x}})=\frac{1}{2}m\dot{\mathbf{x}}^2-V(\mathbf{x}),$ the quantum mechanical amplitude of a particle initially located at position $x_i$ at time $t_i$ to be measured at position $x_f$ at time $t_f$ is given by
\begin{equation}
    \langle x_f,t_f| x_i,t_i\rangle=\int_{\begin{subarray}{l}\{\mathbf{x}:[t_i,t_f]\to \mathbb{R}:\,\mathbf{x}(t_i)=x_i,\mathbf{x}(t_f)=x_f\end{subarray}\}}\mathcal{D}\mathbf{x}\,e^{\frac{i}{\hbar}\int_{t_i}^{t_f}\mathcal{L}\,d\mathbf{t}}\,\;.
\end{equation}
Given an initial wave function $\psi_i=\psi_i(x)$ at time $0,$ we can use this amplitude to determine that the wave function at time $t$ is given by
\begin{equation}\label{wave}
  \psi(x,t)=\int_{\begin{subarray}{l}\{\mathbf{x}:[0,t]\to \mathbb{R}:\,\mathbf{x}(0)=x\end{subarray}\}}\mathcal{D}\mathbf{x}\,e^{\frac{i}{\hbar}\int_{0}^{t}\mathcal{L}\,d\mathbf{t}}\psi_i(\mathbf{x}(t))\;\,.
\end{equation}
A measure theoretic approach to rigorously defining this formula is given by formally replacing $t$ with $it,$ in which case we obtain an expectation value of a random variable with respect to the Wiener measure. We then analytically continue in time to obtain the original path integral (ie. we perform a Wick rotation).
\\\\Feynman, on the other hand, constructed \cref{wave} as a limit of a sequence of approximations involving sums which are similar to Riemann sums, as follows (we simplify to the case that $t=1$)\footnote{Up to a finite normalization constant depending on $V,$ this is the correct formula.}:
\begin{align}\label{approx}
& \nonumber\psi(x,1)\\&=\lim\limits_{N\to\infty}\sqrt{\frac{mN}{2\pi i\hbar}}\int_{-\infty}^{\infty}\prod_{n=1}^{N}dx_n\,\exp{\bigg[\frac{1}{N}\frac{i}{\hbar}\sum_{n=1}^N \frac{m}{2}N^2(x_n-x_{n-1})^2-V(x_n)\bigg]}\psi_i(x_N)\;\,,
\end{align}
where $x_0=x$ and all integration variables are integrated over $(-\infty,\infty).$ 
\\\\In this paper we will formalize this approach, in particular we will provide an answer to the question: what is 
\begin{equation}\label{sum?}
    \frac{1}{N}\sum_{n=1}^N N^2(x_n-x_{n-1})^2-V(x_n)\;?
\end{equation}
This sum is similar to (though distinct from) a Riemann sum, but a priori we don't even have Riemann sums on manifolds, they are defined in a local and coordinate dependent way. What can we replace Riemann sums and \ref{sum?} by when the target is a Riemannian manifold $(M,g)$ or a more general sigma model?   The situation is even more confusing when the target space is more generally a Lie algebroid, eg. the Poisson sigma model. 
\\\\Of course, constructions analogous to Feynman's have been made in more general contexts (lattice QFT, etc.), but we want to be precise.
\\\\To resolve this, we define a notion of \textit{generalized} Riemann sums, which are a coordinate-free notion of Riemann sums. These will be an important part of our construction.
\subsection{The Idea}
Let's briefly explain our approach to understanding \ref{sum?}. To obtain \ref{approx} from \ref{wave} we triangulate $[0,1]$ with $0=t_0<t_1<\ldots<t_N=1$ and let $x_n=\mathbf{x}(t_n).$ We now make the following observation: we have an identification 
\begin{equation}\label{homs}
    \{(x_0,x_1,\ldots,x_{N}): x_0,x_1,\ldots,x_{N}\in\mathbb{R}\}\cong \textup{Hom}(\Delta_{[0,1]},\textup{Pair}\,\mathbb{R})\;.
    \end{equation}
   On the left side is the domain of integration in \ref{approx} and on the right side is the set of morphisms between our triangulation of $[0,1]$ and the pair groupoid of $\mathbb{R}$ (more precisely, $\Delta_{[0,1]}$ is the simplicial set associated to the triangulation).  
   \\\\Traditionally, when doing calculus on manifolds attention is given exclusively to the tangent bundle. Motivated by \ref{homs}, we are going to shift emphasis to the (local) pair groupoid, which integrates the tangent bundle. A large part of this paper is devoted to formalizing calculus on manifolds from the perspective of the pair groupoid, as is needed for our approach to the functional integral.
   \\\\Our main tool used to relate these two perspectives on calculus is a graded version of the van Est map. Originally discovered in the context of Lie groups, one of its first applications was a simpler proof of Lie's third theorem (see \cite{van est}). It was later extended to Lie groupoids by Weinstein–Xu in the context of geometric quantization (\cite{weinstein1}). We develop a graded version to relate cochain data on the pair groupoid (more generally on a Lie groupoid) to data on the tangent bundle (more generally on a Lie algebroid), in a more precise way.\footnote{Actually, there should be a far more general version of the van Est map that applies to higher groupoids, aspects of this are discussed in \cite{Lackman2}.}
   \\\\The pair groupoid is a much simpler geometric structure than the tangent bundle. In particular, the groupoid version of the exterior derivative and differential forms, which we call completely antisymmetric cochains, are much easier to define. There are also completely symmetric cochains, these are groupoid versions of measures and densities — though we will see that they are general enough to even describe the Euler characteristic.
   \subsubsection{The Pair Groupoid's Role in Integration}
   The pair groupoid is a structure which bridges the smooth and simplicial worlds of manifolds:\footnote{In fact, the pair groupoid is a finite dimensional quotient of the singular simpicial set of the manifold.}
   \begin{equation}
\centering\begin{tikzcd}
     & \textup{Pair}\, M \arrow[rd, "VE_{\bullet}"] \arrow[ld, "i^*"'] &   \\
\Delta_M &                                         & TM
\end{tikzcd}
\end{equation}
Here, $\Delta_M$ is the abstract simplicial complex of a triangulation $|\Delta_M|$ of $M,$ and $i:\Delta_M\to\textup{Pair}\,\textup{M}$ is the inclusion map\footnote{We get this map after picking a total ordering of the vertices, though we really only need to pick an orientation.} We can use this map to pull back cochain data defined on $\textup{Pair}\,\textup{M}$ to cochain data on $\Delta_M.$
\\\\The map $VE_{\bullet}$ is a graded version of the van Est map, it computes the Taylor expansion\footnote{More correctly, it computes the jet.} of cochain data on $\textup{Pair}\,\textup{M}$ at the identity bisection, along the source fibers; the result is cochain data on the tangent bundle $TM.$ In degree $0$ it agrees with the map defined by Weinstein–Xu, commonly denoted $VE$ (up to antisymmetrization).
\\\\We obtain a generalized Riemann sum of a top form $\omega$ on $M$ by lifting it to an $A_{n+1}$-invariant cochain $\Omega$ on $\textup{Pair}\,\textup{M}$ and computing 
\begin{equation}\label{exr}
    \sum_{\Delta\in \Delta_M}\iota^*\Omega(\Delta)\;.
\end{equation}
We then take a limit over triangulations to get $\int_M\omega\,.$ If $\Omega$ is a cocycle then $\ref{exr}$ is the cap product of $\iota^*\Omega$ with the fundamental class of $M,$ in which case \ref{exr} is exactly equal to $\int_M\omega\,,$ for any triangulation. This is what occurs in the fundamental theorem of calculus.
\\\\As a simple example of this, let $\omega=f\,dx$ on $\mathbb{R}.$ The map
\begin{equation}\label{coch1}
    (x,y)\mapsto f(x)(y-x)
\end{equation}
defines a $1$-cochain $\Omega$ on $\textup{Pair}\,[0,1]$ such that $VE_0(\Omega)=f\,dx\,.$ The abstract simplicial complex $\Delta_{[0,1]}$ is determined by points $0=x_0<x_1<\ldots<x_{n}=1.$ The generalized Riemann sum \ref{exr} is then the left hand Riemann sum
\begin{equation}
    \sum_{i=0}^{n-1}f(x_i)(x_{i+1}-x_i)\;.
\end{equation}
As an example, from our perspective the cochain
\begin{equation}
    (x,y)\mapsto f(x)(y-x)+g(y)(y-x)^2
\end{equation}
is just as good as \ref{coch1}, for any smooth $g$ (since they have the same first order Taylor expansions at $y=x$) and therefore may be used to give the generalized Riemann sum
\begin{equation}
    \sum_{i=0}^{n-1}f(x_i)(x_{i+1}-x_i)+g(x_{i+1})(x_{i+1}-x_i)^2\;.
\end{equation}
These both equal $\int_0^1f\,dx$ in the limit over triangulations of $[0,1].$ On the other hand, if $dF=f\,dx,$ then using the cocycle 
\begin{equation}
    (x,y)\mapsto F(y)-F(x)
\end{equation}
in the generalized Riemann sum produces the exact value of the integral.
\\\\It is easy to see how the Riemann-Stieltjes integral $\int_0^1 f\,dg$ fits into this: $(x,y)\mapsto f(x)(g(y)-g(x))$ is a $1$-cochain on the pair groupoid. We will use this fact when discussing stochastic integrals and to generalize the Riemann-Stieltjes integral to manifolds.
\subsubsection{The Fundamental Theorem of Calculus on Manifolds}
We will use these concepts to state a generalization of the fundamental theorem of calculus on manifolds, for which part 1 says the following:
\begin{theorem*}
Let $M$ be an $n$-dimensional manifold and let $(\Omega_M,\Omega_{\partial M})\in H^{n}\textup{Pair}\,(M,\partial M)_{\textup{loc}}\,.$ Then
\begin{equation}\label{fundi}
[M]\frown (\Omega_M,\Omega_{\partial M})=\int_M VE(\Omega_M)-\int_{\partial M}VE(\Omega_{\partial M})\;.
\end{equation}
\end{theorem*}
Here, $ H^{n}\textup{Pair}\,(M,\partial M)_{\textup{loc}}$ is the cohomology of the local pair groupoid, relative to its boundary; $(\Omega_M,\Omega_{\partial M})$ pulls back to a class in $H^n(M,{\partial M};\mathbb{R})$ via the projection of a simplex onto its set of vertices – on the left side of \ref{fundi} we have its cap product with the fundamental class. This result contains Stokes' theorem as a special case.
\\\\As a very simple example of the above, consider the cocycle $\Omega(e^{i\theta_0},e^{i\theta_1})=\theta_1-\theta_0$ on the local pair groupoid of $S^1.$ We have that $VE(\Omega)=d\theta,$ and of course both sides of \cref{fundi} equal $2\pi.$
\\\\Our generalization will be used to prove \cref{funco}, which is the main result we have to justify our approach to functional integration. In particular, we will use this corollary to show that our construction of functional integrals reproduces Dijkgraaf–Witten's and Freed–Quinn's partition function for Chern-Simons for finite gauge groups (\cite{witten}, \cite{freed}), in which case the target is the classifying space $BG.$
\subsubsection{Graded van Est Map}
To briefly explain why a graded version of the van Est map is necessary: it is because of regularity issues. If the paths in $C([0,1],\mathbb{R})$ were generically Lipschitz continuous then using a graded version wouldn't be necessary. However, in Brownian motion and Feynman's path integral, the paths have H\"{o}lder exponent in $(1/3,1/2)$ on a set of full measure. Thus, we have to compute higher order Taylor expansions. 
\\\\The difference in Taylor expansions at second order is what leads to the different stochastic integrals, eg. the It\^{o} and Stratonovich integrals, which classically are the same but differ with respect to Wiener space (ie. the It\^{o} integral uses a left hand Riemann sum and the Stratonovich integral uses the average of the left and right hand Riemann sums). This leads to ambiguities when defining the path integral.
\subsubsection{Approximating Functional Integrals and Examples}
The framework we develop can be applied to functional integrals arising from any classical field theory valued in a tangent bundle (or Lie algebroid), though we won't discuss general convergence results or renormalization.  Our main justification for this approach is \cref{funco}. The idea is to
\begin{enumerate}
    \item triangulate the domain, 
    \item integrate the tangent bundle (or Lie algebroid) to the local pair groupoid (or Lie groupoid),
    \item approximate the maps between the domain and target by morphisms of simplicial sets,
    \item integrate the cochain data via $VE_{\bullet}\,,$
    \item form the generalized Riemann sums,
    \item define a measure on the hom spaces by using available data (eg. a Riemannian metric, symplectic form, Haar measure),
    \item compute the approximations of the functional integral,
    \item take a limit over triangulations.\footnote{We will discuss viewing the limit as a kind of improper integral.}
\end{enumerate}    
In particular, we will discuss examples obtained from the Poisson sigma model construction of Kontsevich's deformation quantization formula (\cite{catt}, \cite{kontsevich}). In this case, there is overlap with the data required to construct a geometric quantization in the sense of Hawkins and Weinstein (\cite{eli}, \cite{weinstein}); both require a symplectic groupoid and a cocycle. However, we don't need polarizations or connections, or any ad hoc quantization maps. We will also discuss some other topological quantum field theories.
\\\\We expect that, in general, our construction of the path integral will converge to the correct quantity when the domain is a one dimensional manifold and the target is a compact Riemannian manifold — finite dimensional approximations to Wiener's path integral in this context were obtained by Andersson and Driver in \cite{lars} (also see \cite{samp}), but we won't go into this level of generality in this paper.
\subsection{Outline of Paper}
In the introduction we gave an overview of the paper. In the appendix we have included most of the basic Lie groupoid theory which is relevant to this paper. The examples relevant to Brownian motion are the simplest ones: the pair groupoid and the tangent bundle. 
\\\\We will begin by stating the main result and outlining the idea.
In next several sections we will be formulating basic differential geometry using the pair groupoid and comparing it with standard treatment of differential geometry on the tangent bundle. We will begin by defining the analogue of differential forms and measures, followed by a generalization of the van Est map (but still not the fully general one). We will then define generalized Riemann sums and state the fundamental theorem of calculus on manifolds. After this, we will discuss observables, followed by the graded van Est map. We will then define a  Riemann-Stieltjes integral on manifolds and prove a convergence theorem. \\\\Finally, we will further discuss the groupoid construction of some functional integrals in addition to Feynman's and Wiener's. In particular, these include Chern-Simons for finite groups and the Poisson sigma model on the disk. For the latter, we are mostly concerned with its relation to Kontsevich's solution of the formal deformation quantization problem of Poisson manifolds and obtaining non-perturbative deformation quantizations.
\\\\Sections marked with $*$ contain important information but are not mandatory on a first reading for a basic understanding of the framework. We have included an index of notation at the end, see \ref{a4}.
\subsubsection{Convention for Wedge Products}\label{warning}
The convention we use for wedge products is that we don't multiply by factorials after antisymmetrizing, eg. we define
\begin{equation}
    dx\wedge dy=\frac{1}{2}(dx\otimes dy-dy\otimes dx)\;.
\end{equation}
We do this because we partition domains of integration into simplices, not parallelpipeds, and with this definition $dx^1\wedge\cdots\wedge dx^n$ gives the volume of the standard $n$-simplex when evaluated on $(\partial_{x^1},\ldots,\partial_{x^n}).$ With respect to integration, it's still true that
\begin{equation}
    \int_M dx^1\wedge\cdots\wedge dx^n=\int_M dx^1\cdots dx^n\;,
\end{equation}
so we have 
\begin{equation}
     \int_{\Delta^n} dx^1\wedge\cdots\wedge dx^n=\frac{1}{n!}\;.
\end{equation}
\subsection*{Acknowledgements}
I'd like to thank David Pechersky for discussions about Brownian motion.
\section{Main Result}
Here we state and prove the main result, We will go on to explain how we can interpret the Wiener space as a direct limit of probability spaces.
\subsection{Statement and Proof}
We will be assuming standard facts about Brownian motion, for a textbook treatment see \cite{rick}, \cite{peter}. We will state a simple version of the result, we aren't optimizing. In addition, we expect it to generalize to the case where a magnetic potential is added to the action (or any stochastic integral), and to the case where the target is a compact Riemannian manifold.
\\\\In the following, $VE{_\bullet}$ is a generalization of the van Est map whose image contains cochains defined on a Lie algebroid, but not just antisymmetric ones. In addition, $(\mathbb{R},dx^2)$ is standard Euclidean space; $dt^2$ is the standard metric on $[0,1];$ $\,\textup{Pair}^{(k)}(\mathbb{R})=\mathbb{R}^{k+1}$ is the manifold in degree $k$ of the nerve of the pair groupoid.
\begin{theorem}\label{main}
Let $V:\mathbb{R}\to\mathbb{R}$ be a smooth function which is bounded from below, let $f$ be a smooth, bounded $k$-cochain on $\textup{Pair}\,\mathbb{R}$ and let $t_{i_0}\le t_{i_1}\le\cdots \le t_{i_k}\in [0,1].$ Define an observable by
\begin{equation}
      f_{t_{i_0},\ldots,t_{i_k}}:C([0,1],\,\mathbb{R})\to\mathbb{R},\, \;\;f_{t_{i_0},\ldots,t_{i_k}}(\mathbf{x})=f(\mathbf{x}(t_{i_0}),\ldots,\mathbf{x}(t_{i_k}))\,.
\end{equation}
Then
\begin{align}\label{maint}
   &\nonumber\mathbb{E}_{\mu_W}[e^{-\int_{0}^1 V(\mathbf{x}(t))\,dt}\,f_{t_{i_0},\ldots,t_{i_k}}] 
   \\&=\lim\limits_{\Delta_{[0,1]}\in\mathcal{T}_{[0,1]}}\int_{\begin{subarray}{l}\{\mathbf{x}\in\textup{hom}(\Delta_{[0,1]},\, U_{\Delta }):\, \mathbf{x}(0)=0\end{subarray}\}}\mathcal{D}\mathbf{x}\exp^{-S[\mathbf{x}]}f_{t_{i_0},\ldots,t_{i_k}}(\mathbf{x})\,\;,
\end{align}
where $\mu_W$ is the Wiener measure. 
\\\\The limit is taken in the sense of nets, over all triangulations $\Delta_{[0,1]}\in\mathcal{T}_{[0,1]}$ which have $t_{i_0},\ldots,t_{i_k}$ as vertices. Here, $U_{\Delta }$ depends on the triangulation. Letting $0=t_0<t_1<\ldots t_n=1$ denote its vertices and letting $\Delta t=\textup{sup}_i|t_{i+1}-t_i|\,,$ we have that
\begin{equation}
    U_{\Delta }= \big\{(x_0,x_1)\in \textup{Pair}\,\mathbb{R}:\;|x_{1}-x_0|<2\sqrt{\Delta t|\log{\Delta t}|}\,\big\}\;.\footnote{See remark \ref{loca}.}
\end{equation}
\end{theorem}
$\,$\\We describe the meaning of \ref{maint} below: consider the smooth data $\mathcal{G}_{dx},\mathcal{G}_{dt},\mathcal{V},$ where
\begin{enumerate}\label{datas}
    \item $\mathcal{G}_{dx},\, \mathcal{V}$ are $1$-cochains on $\textup{Pair}\,\mathbb{R}\,,$
    \item $\mathcal{G}_{dt}$ is a $1$-cochain on $\textup{Pair}\,[0,1]\,,$
\end{enumerate}
such that
\begin{enumerate}
    \item $VE{_3}(\mathcal{G}_{dx})=dx^2,$\footnote{ie. $dx$ defines a pointwise linear map, and $dx^2=(dx)^2$ is its pointwise square.}
    \item $VE{_1}(\mathcal{G}_{dt})=dt\,,$
    \item $VE_{-1}(\mathcal{V})=V$ (ie. $\mathcal{V}\vert_{\mathbb{R}}=V$),

\end{enumerate}
such that the fifth, third, and first (covariant) derivatives of $\mathcal{G}_{dx},\, \mathcal{G}_{dt},\,\mathcal{V}$ along the source fibers are bounded, respectively. 
\\\\Let $\mathcal{T}_{[0,1]}$ be the set of triangulations of $[0,1],$ partially ordered by common refinement; denote its elements by $\Delta_{[0,1]}$ and denote by $\Delta^1$ the non-degenerate $1$-simplices in $\Delta_{[0,1]}.$ Then, for \begin{equation*}
    \mathbf{x}\in\textup{hom}(\Delta_{[0,1]}, \textup{Pair}\,\mathbb{R})\,,
    \end{equation*}
    we define
\begin{equation}\label{actions}
    S[\mathbf{x}]= \sum_{\Delta^1\in [0,1]_{\Delta} }\frac{1}{2}\frac{\mathcal{G}_{dx}(\mathbf{x}(\Delta^1))}{\mathcal{G}_{dt}(\Delta^1)}+\mathcal{V}(\mathbf{x}(\Delta^1))\mathcal{G}_{dt}(\Delta^1)\,\;.
\end{equation}
\\\\Here, $\mathcal{D}\mathbf{x}$ is a measure which depends on $\Delta_{[0,1]}$ and is defined as follows: let $\Delta_{[0,1]}$ be a triangulation with $m+1$ vertices. Then 
\begin{align}
    &\nonumber\{\mathbf{x}\in\textup{hom}(\Delta_{[0,1]},\textup{Pair}\,\mathbb{R}):\,\mathbf{x}(0)=0\}
    \\&=\{(x_0,\ldots,x_m)\in\textup{Pair}^{(m)}(\mathbb{R}):x_0=0\}=\mathbb{R}^{m}
\end{align}
and $\mathcal{D}\mathbf{x}$ is then the product measure induced by the metric $dx^2.$
\\\\Finally, $f_{t_{i_0},\ldots,t_{i_k}}(\mathbf{x})=f(\mathbf{x}(t_{i_0}),\ldots,\mathbf{x}(t_{i_k})).$
\\
\begin{remark}\label{loca}
This is a local pair groupoid. In a sense, we're doing the opposite of what is traditionally done: normally the entire pair groupoid is used as the target, whereas we are really only using its germ near the identity bisection.\footnote{Of course, the germ is much more similar to the tangent bundle than the whole groupoid is.} Of course, we can always absorb this dependency into the measure on the corresponding hom space.
\\\\There are other prescriptions for the target we can use. However, shrinking the neighborhoods with the triangulation in some appropriate way is the most natural thing to do, especially when considering that Feynman's construction contains discontinuous maps, see chapter 7.1 of \cite{rick}. We need to at least use the local groupoid so that the approximations converge.
\end{remark}
\begin{proof}(\cref{main})
Let $F:\mathbb{R}^n\to\mathbb{R}$ be bounded and Borel measurable, let $0< t_{1}\le \cdots\le t_{n}\le 1$ and consider the map 
\begin{equation}
    \mathbf{x}\mapsto F_{t_{1},\ldots,{t_{n}}}(\mathbf{x})= F(\mathbf{x}(t_{1}),\ldots,\mathbf{x}(t_{n}))\;.
\end{equation}
Using the finite dimensional distributions of the Wiener process,\footnote{See \cite{lars}, \cite{peter}.}we have that 
\begin{align}\label{find}
&\nonumber\mathbb{E}_{\mu_W}[ F_{t_{1},\ldots,{t_{n}}}]=
\\&\frac{1}{\sqrt{2\pi(t_{{j+1}}-t_{{j}}})}\int_{-\infty}^{\infty}\prod_{j=1}^{n}dx_j\,\exp\bigg[-\frac{1}{2}\sum_{j=1}^{n-1}\frac{(x_{j+1}-x_j)^2}{(t_{{j+1}}-t_{{j}})}\bigg]F(x_1,\ldots,x_n)\;.
\end{align}
$\,$\\Now, let $0=t_0<t_1<\cdots<t_n=1$ be a triangulation of $[0,1]$ containing $i_0,\ldots,i_k$ as vertices and let 
\begin{equation}
    \Delta x_i=x_{i+1}-x_i,\;\; \;\Delta t_i=t_{i+1}-t_i\;.
    \end{equation}
Computing \ref{actions} with the cochain data \ref{datas} gives 
\begin{align}\label{mydat}
    &\sum_{i=0}^{n-1}\frac{1}{2}\frac{\Delta x_i^2}{\Delta t_i}\frac{1+h_{dx}(x_i,x_{i+1})\Delta x_i^3}{1+h_{dt}(t_i,t_{i+1})\Delta t_i^2} 
    \\&\nonumber +V(x_{i})(\Delta t_i+h_{dt}(t_i,t_{i+1})\Delta t_i^3)+h_{V}(x_i,x_{i+1})\Delta x_i(\Delta t_i+h_{dt}(t_i,t_{i+1})\Delta t_i^3)
\end{align}
where $h_{dx}(x_i,x_{i+1}),\, h_{dt}(t_i,t_{i+1}),\, h_{V}(x_i,x_{i+1})$ are the remainders derived from using Taylor's theorem on $\mathcal{G}_{dx}, \mathcal{G}_{dt}, \mathcal{V}$ along the source fibers (and are bounded by assumption). 
\\\\Now consider \ref{find} with the random variable
\begin{equation}\label{rand}
    F(x_1,\ldots,x_n)= \mathbbm{1}_{A}(x_1,\ldots,x_n)\textbf{}e^{-U(x_1,\ldots,x_n)} f(x_{i_0},\ldots,x_{i_k})\;,
\end{equation}
where
\begin{align}
& A=\left\{ |x_{j+1}-x_j|<2\sqrt{(t_{j+1}-t_j)|\log{(t_{j+1}-t_j)|}}\,,\;0\le j\le n-1\right\}\;,
     \\\nonumber\\ &\nonumber U(x_1,\ldots,x_n)= \sum_{j=0}^{n-1}\frac{1}{2}\frac{\mathcal{G}_{dx}(x_j,x_{j+1})}{\mathcal{G}_{dt}(t_j,t_{j+1})}+\mathcal{V}(x_j,x_{j+1})\mathcal{G}_{dt}(t_j,t_{j+1})- \sum_{j=0}^{n-1}\frac{1}{2}\frac{(x_{j+1}-x_{j})^2}{t_{j+1}-t_j}\;.
\end{align}
Using \ref{mydat} together with L\'{e}vy's modulus of continuity, ie. Wiener paths almost surely satisfy 
\begin{equation}
   |\mathbf{x}(t')-\mathbf{x}(t)|<2\sqrt{(t'-t)|\log{(t'-t)|}}
    \end{equation}
     for all small enough $t'-t>0,$\footnote{See \cite{peter}.} we see that 
     \begin{equation}
         \mathbf{x}\mapsto F(\mathbf{x}(t_1),\ldots,\mathbf{x}(t_n))
         \end{equation}
    converges almost surely to
\begin{equation*}
    \mathbf{x}\mapsto e^{-\int_{0}^1 V(\mathbf{x}(t))\,dt}\,f_{t_{i_0},\ldots,t_{i_k}}(\mathbf{x})\;.
\end{equation*}
The result now follows by applying the dominated convergence theorem. 
     \end{proof}
     
\subsection{*Sigma Algebras and Approximations}\label{impropb}
Here we will discuss $\sigma$-algebras and the idea of approximating the space of continuous maps by a hom space between simplicial spaces. This isn't necessary for understanding the main theorem (as the $*$ indicates), but it does motivate the idea. We will discuss this more generally in \cref{imrop}.
\\\\The $\sigma$-algebra associated to the Wiener measure is the set of continuous maps $C([0,1],\mathbb{R})$ equipped with the Borel $\sigma$-algebra $\mathcal{B}$ generated by the compact-open topology, ie. the topology generated by the sup norm.
\\\\The set of triangulations $\mathcal{T}_{[0,1]}$ of $[0,1]$ form a directed set ordered by refinement. For each triangulation $\Delta_{[0,1]}\in \mathcal{T}_{[0,1]}$  we have a sub-$\sigma$-algebra $\mathcal{B}_{\Delta_{[0,1]}}\subset\mathcal{B}$ defined as the finest sub-$\sigma$-algebra of $\mathcal{B}$ for which homotopy classes of maps relative to the vertices of $\Delta_{[0,1]}$ are atoms\footnote{A measurable set in a $\sigma$-algebra is an atom if it contains no proper nonempty measurable sets.} 
\\\\Explicitly, $B\in \mathcal{B}_{\Delta_{[0,1]}}$ if and only if $B\in\mathcal{B}$ and $f, g$ being homotopic relative to the vertices of $\Delta_{[0,1]}$ implies that they are either both in $B$ or neither are. This is the $\sigma$-algbera generated by pointwise evaluation at the vertices of $\Delta_{[0,1]}.$ This $\sigma$-algebra is precisely the Borel $\sigma$-algebra of continuous maps 
\begin{equation}
    \textup{Hom}(\Delta_{[0,1]},\textup{Pair}\,\mathbb{R})\;.
    \end{equation}
This is similar to a filtration of $\sigma$-algebras, except that we don't have a total order.
\\\\The $\sigma$-algebras $\{\mathcal{B}_{\Delta_{[0,1]}}\}_{\Delta_{[0,1]}\in \mathcal{T}_{[0,1]}}$ approximate $\mathcal{B}.$ Indeed, 
\begin{equation}  
\Delta_{[0,1]}\le\Delta_{[0,1]}' \implies \mathcal{B}_{\Delta_{[0,1]}}\subset \mathcal{B}_{\Delta_{[0,1]}'}
\end{equation}
and we have a direct system in the category of $\sigma$-algebras on $C([0,1],\mathbb{R}),$ with morphisms the inclusions. The direct limit\footnote{This is equivalent to an inverse system and inverse limit in the category of $\sigma$-algebras on $C([0,1],\mathbb{R}),$ with the only allowable morphism the identity map.} is  $\mathcal{B}.$
\\\\The idea is to now equip each $\sigma$-algebra $\mathcal{B}_{\Delta_{[0,1]}}$ with a measure $\mu_{\Delta_{[0,1]}}$ which approximates the Wiener measure $\mu_W$ on $\mathcal{B}_{\Delta_{[0,1]}}.$ In particular, we want the following: if $B\in \mathcal{B}_{\Delta_{[0,1]}}$ then
\begin{equation}\label{cond}
       \lim\limits_{\Delta_{[0,1]}'\ge \Delta_{[0,1]}}  \mu_{\Delta_{[0,1]}'}(B)=\mu_W(B)\;.
\end{equation}
where the limit is taken in the sense of nets. 
\subsubsection{*Connection With Feynman's Construction}
We are going to discuss \ref{main} with Feynman's choice of cochain data. For brevity we assume $V=0.$ 
\\\\We can choose $\mathcal{G}_{dt}, \mathcal{G}_{dx}$ in \cref{main} to be given by  
\begin{equation}\label{feynd}
    \mathcal{G}_{dt}(t_0,t_1)=(t_1-t_0)^2,\, \;\;\mathcal{G}_{dx}(x,y)=(y-x)^2\;.
\end{equation}
Now, choose a triangulation $\Delta_{[0,1]}$ of $[0,1]$ with vertices $0=t_0<t_1<\ldots<t_n=1.$ Then given
\begin{equation}
    \mathbf{x}\in \text{hom}(\Delta_{[0,1]},\textup{Pair}\,\mathbb{R}),\,\;\;\mathbf{x}(t_i,t_{i+1})=(x_i,x_{i+1}),\;0\le i\le n-1,
\end{equation}\cref{actions} is equal to
\begin{equation}
S[\mathbf{x}]=\sum_{0\le i\le n-1 }\frac{1}{2}\frac{(x_{i+1}-x_i)^2}{t_{i+1}-t_i}\;,
\end{equation}
which should remind the reader of 
\begin{equation}
    S[\mathbf{x}]=\int_0^1 \frac{1}{2}\Big(\frac{d\mathbf{x}}{dt}\Big)^2\,dt\;.
\end{equation}
\subsubsection{*The Wiener Measure as a Direct Limit of Approximations}
Using the choices in \ref{feynd}, the approximation (referred to in \cref{cond}) to the probability density of a Brownian particle to travel from point $x$ at time $t_0$ to point $y$ at time $t_1$ is given by
\begin{equation}\label{densi}
   p(x,t_0;y,t_1)= \frac{1}{\sqrt{2\pi(t_1-t_0)}}e^{-\frac{1}{2}\frac{(y-x)^2}{t_1-t_0}}.
\end{equation}
This approximation is actually exact. This quantity has the special property that 
\begin{equation}\label{conv}
     p(x,t_0;z,t_2)=\int_{-\infty}^{\infty}p(x,t_0;y,t_1)p(y,t_1;z,t_2)\,dy\,.
\end{equation}
The integral above is the convolution of $p(\cdot, t_0;\cdot, t_1)$ with $p(\cdot, t_1;\cdot, t_2),$ with respect to $\text{Pair}\,\mathbb{R}.$ 
\\\\For each $x,t_0,t_1$ we have a measure $\mu_{[t_0,t_1]}$ obtained by integrating $y$ in \ref{densi} over subsets of $\mathbb{R}.$ Given a triangulation $\Delta_{[0,1]}$ of $[0,1]$ this induces a measure $\mu_{\Delta_[0,1]}$ on $\mathcal{B}_{\Delta_{[0,1]}},$ referred to in \cref{impropb}. This is the usual measure in the finite dimensional approximations of the path integral.
\\\\We can formalize taking the limit as follows: consider the subcategory 
\begin{equation}
    \text{Pair}_{\le}\,[0,1]\xhookrightarrow{} \text{Pair}\,[0,1]
    \end{equation}
consisting only of oriented arrows, ie. $(t_0,t_1)$ with $t_0\le t_1.$ For a Lie groupoid $G\rightrightarrows X,$ let $\mathcal{M}_s(G)$ be the space of Borel probability measures along the source fibers. There is a projection map
\begin{equation}
    \mathcal{M}_s(G)\to X\;,
\end{equation}
and sections $\Gamma(\mathcal{M}_s(G))$ of this map form a monoid with respect to convolution, ie. convolution is associative and has identity.
\\\\\Cref{conv} then implies that, for these choices of $\mathcal{G}_{dt},\,\mathcal{G}_{dx},$ the induced map
\begin{equation}
    \mu_{[\cdot,\cdot]}:\text{Pair}_{\le}([0,1])\to\Gamma(\mathcal{M}_s(\text{Pair}(\mathbb{R})))
\end{equation}
is a $\Gamma(\mathcal{M}_s(\text{Pair}(\mathbb{R})))$-valued cocycle, ie.
\begin{equation}\label{coc}
    \mu_{[t_0,t_2]}=\mu_{[t_0,t_1]}\ast \mu_{[t_1,t_2]}\,.
\end{equation}
It is this property\footnote{This property is closely related to the semigroup property of the heat kernel.} that makes the probabilities determined by the approximation exact. That is, letting $\mu_{\Delta_{[0,1]}}$ denote the measure on $\mathcal{B}_{\Delta_{[0,1]}}$ induced by $\mu_{[\cdot,\cdot]},$ we have that
\begin{equation}
  \mu_{\Delta_{[0,1]}}=\mu_W
\end{equation}
on $\mathcal{B}_{\Delta_{[0,1]}}$ (see \ref{cond}). Therefore, with these choices the limit in \ref{cond} can be taken in a categorical sense, ie. we have the following:
\begin{proposition}\label{direct}
Let $\mathcal{G}_{dx},\mathcal{G}_{dt}$ be as in \ref{datas}. Suppose that the induced map $\mu_{[\cdot,\cdot]}$ is a cocycle in the sense of \ref{coc}.
Then the Wiener space is the direct limit of the probability spaces \begin{equation}
\{(C([0,1],\mathbb{R}),\mathcal{B}_{\Delta_{[0,1]}},\mu_{\Delta_{[0,1]}})\}_{\Delta_{[0,1]}\in\mathcal{T}_{[0,1]}}\,\;.
\end{equation}
\end{proposition}
$\,$\\Evidently, these choices of data are special. In general we won't have this, we will only get the Wiener measure in the limit in the sense of analysis.
\subsubsection{*The Lebesgue Measure as a Direct Limit of Finite Sigma Algebras}\label{lebe}
This construction of the Wiener measure in the previous subsection is analogous to the following construction of the Lebesgue measure as a direct limit: consider the interval $[0,1]$ with the Borel $\sigma$-algebra $\mathcal{B},$ and for each triangulation $\Delta_{[0,1]}$ let $\mathcal{B}_{\Delta_{[0,1]}}$ be the finest sub-$\sigma$-algebra such that the open faces are atoms, ie. the vertices and corresponding open intervals are measurable and don't contain any proper nonempty measurable subsets. Consider the map 
\begin{equation}
    \Omega:\text{Pair}\,[0,1]\to\mathbb{R},\;\Omega(x,y)= y-x\,.
\end{equation}
This is a cocycle such that $VE_0(\Omega)=dx.$ Using the orientation on $[0,1]$ gives us a direct system of measures spaces 
\begin{equation}
   \{ ([0,1],\mathcal{B}_{\Delta_{[0,1]}},\Omega_{\Delta_{[0,1]}})\}_{\Delta_{[0,1]}\in \mathcal{T}_{[0,1]}}
\end{equation}
whose limit is the Lebesgue measure, eg. if $(a,b)$ is a face of $\Delta_{[0,1]}$ then $\Omega_{\Delta_{[0,1]}}(a,b)=b-a.$ To make the analogy even tighter, in \cref{direct} we can (informally) say that \begin{equation}
    VE_0(\mu_{[\cdot,\cdot]})=\frac{d}{dx^2},
\end{equation} as a distribution.
\\\\If we instead choose a $1$-cochain $\Omega$ which isn't a cocycle, but such that $VE_0(\Omega)=dx,$ we could still derive the Lebesgue measure by using approximate measures, as in \ref{cond}.
\section{Cochains}
In this section we will define the nerve of a groupoid, followed by completely symmetric/antisymmetric cochains on groupoids and algebroids, and finally local groupoids. Our description of the nerve of a groupoid will differ from the traditional one, eg. in \cite{Crainic}.\footnote{In the context of the theory of rough paths, germs of cochains on the pair groupoid were also used in \cite{step} (without using this terminlogy).}
\\\\We have included an index of notation at the end of this paper, see \ref{a4}.
\subsection{The Nerve of a Groupoid and Cochains}
We are going to give a definition of the nerve of a groupoid that is slightly different, but equivalent to the usual definition. Our definition makes analogies between Lie groupoids and Lie algebroids clearer, and as a result it makes the van Est map easier to define. 
\begin{definition}
The nerve of a Lie groupoid $G\rightrightarrows X$, denoted $G^{(\bullet)},$ is a simplicial manifold\footnote{We describe the face and degeneracy maps in the next section.} which in degree $n\ge 1$ is given by the following fiber product:
\begin{equation}\label{second}
    G^{(n)}=\underbrace{G\sideset{_s}{_{s}}{\mathop{\times}} G \sideset{_s}{_{s}}{\mathop{\times}} \cdots\sideset{_s}{_{s}}{\mathop{\times}} G}_{n \text{ times}}\,.
    \end{equation}
We set $G^{(0)}=X.$ In the context of $n$-cochains (soon to be defined), we will frequently identify $x\in X$ with $(\textup{id}(x),\ldots,\textup{id}(x))\in G^{(n)}$; we call the image of $X$ in $G^{(n)}$ the identity bisection.
\end{definition}
$\,$\\The traditional definition of the nerve sets
\begin{equation}
    G^{(n)}=\underbrace{G\sideset{_t}{_{s}}{\mathop{\times}} G \sideset{_t}{_{s}}{\mathop{\times}} \cdots\sideset{_t}{_{s}}{\mathop{\times}} G}_{n \text{ times}}\,,
\end{equation}
which is the space consisting of $n$ arrows which are sequentially composable. The two definitions are equivalent, with isomorphism
\begin{equation}
    (g_1,g_2,\ldots,g_n)\in {G\sideset{_t}{_{s}}{\mathop{\times}} \cdots\sideset{_t}{_{s}}{\mathop{\times}} G}_{}\mapsto (g_1,g_1g_2,\ldots,g_1g_2\cdots g_n)\in {G\sideset{_s}{_{s}}{\mathop{\times}}\cdots\sideset{_s}{_{s}}{\mathop{\times}} G}_{}\,.
\end{equation}
Our definition of $G^{\bullet}$ makes it clear that the symmetric group $S_n$ acts on $G^{(n)}$ by permutations. A little less obvious is that $S_{n+1}$ acts on $G^{(n)}.$ We define $S_n$ to be permutations of the set $\{0,1,\ldots,n-1\}.$
\begin{definition}
For $\sigma\in S_{n+1}$ and for $(g_1,\ldots,g_n)\in G^{(n)},$ we let
\begin{equation}
    \sigma\cdot(g_1,\ldots,g_n):=(g^{-1}_{\sigma^{-1}(0)}g_{\sigma(1)},\ldots,g^{-1}_{\sigma^{-1}(0)}g_{\sigma(n)})\,,
\end{equation}
where $g_0:=\textup{id}({s(g_1)}).$ The result is a point in $G^{(n)}$ whose common source is $t(g_{\sigma^{-1}(0)}).$
\end{definition}
With this definition, if $\sigma$ fixes $0$ then we get the obvious permutation action of $S_n$ (the action of $S_3$ was used in \cite{hoyo} and an alternative construction on groups is done in the appendix of \cite{yag}).
\begin{exmp}
Consider the pair groupoid $\textup{Pair}\,\textup{X}\rightrightarrows X$ (defined in \ref{pair}). We have that 
\begin{equation}
    X^{(n)}\cong X^{n+1}=\{(x_0,x_1,\ldots,x_n): x_i\in X\}\;.\end{equation}
The $n$ arrows are $(x_0,x_1),(x_0,x_2),\ldots,(x_0,x_n)$; $S_n$ acts by permuting the coordinates $x_1,\ldots ,x_n$ while $S_{n+1}$ acts by permuting all coordinates. In the context of $n$-cochains we will identify $x\in X$ with $(x,\ldots,x)\in X^{n+1}.$
\end{exmp}
The following definition is standard, eg. see \cite{weinstein}:
\begin{definition}\label{normc}
A (smooth) $n$-cochain is a (smooth) function $\Omega:G^{(n)}\to\mathbb{R}$ on a groupoid $G\rightrightarrows X.$ It is said to be normalized if $g_i$ being an identity for some $1\le i\le n$ implies that $\Omega(g_1,\ldots,g_n)=0.$ By convention we consider $0$-cochains to be normalized without further condition. We denote the sets of $n$-cochains and normalized $n$-cochains by 
\begin{align*}
   & C^n(G)=\text{ space of smooth n-cochains}\;,
    \\& C_0^n(G)=\text{ space of normalized smooth n-cochains}\;.
\end{align*}
\end{definition}
Now since we have an action of $S_{n+1}$ on $G^{(n)},$ we get an action of $S_{n+1}$ on $n$-cochains $\Omega$ by duality, ie.
\\\begin{equation}
(\sigma\cdot \Omega)(g_1,\ldots,g_n)=\Omega(\sigma^{-1}\cdot(g_1,\ldots,g_n))\,.
\end{equation}
\begin{definition}
A $n$-cochain $\Omega$ antisymmetric (symmetric) if it is antisymmetric (symmetric) with respect to the action of $S_n.$ 
\end{definition}
Normalized, antisymmetric cochains behave much like $n$-forms on the corresponding Lie algebroid (soon to be defined), however the following cochains are really the right analogue due to their behaviour with respect to the groupoid differential and their connection to Stokes' theorem. In particular, they are automatically normalized:
\begin{definition}\label{antic}
A cochain $\Omega:G^{(n)}\to\mathbb{R}$ is completely antisymmetric if it is antisymmetric with respect to the action of $S_{n+1}.$ We use the notation
\begin{align*}
  \Lambda^nG^*=\text{ subspace of completely antisymmetric n-cochains}\;.
\end{align*}
\end{definition}
There is a natural graded product on antisymmetric cochains:
\begin{definition}\label{wedge}
Let $\Omega\in   \Lambda^iG^*,\;\Omega'\in \Lambda^jG^*.$ Then we get a cochain $\Omega\wedge\Omega'\in   \Lambda^{i+j}G^*$ defined by completely antisymmetrizing the map
\begin{equation}
    (g_1,\ldots,g_i,g_{i+1},g_{i+j})\mapsto \Omega(g_1,\ldots,g_i)\Omega'(g_{i+1},\ldots,g_{i+j})\;.
\end{equation}
\end{definition}
$\,$\\Completely antisymmetric cochains vanish on degenerate points in $G^{(n)},$ ie. points $(g_1,\ldots,g_n)$ where either there is a $g_i$ which is an identity or if there are $g_i, g_j$ such that $i\ne j$ but $g_i=g_j.$ This is also true for normalized cochains which are symmetric (instead of antisymmetric) under $S_{n+1},$ but for cochains which aren't normalized we will add an additional condition:
\begin{definition}\label{well}
A cochain $\Omega:G^{(n)}\to\mathbb{R}$ is completely symmetric if it is symmetric with respect to the action of $S_{n+1}$ and if it is well-defined on degeneracies in the following sense: if $(g_1,\ldots, g_n),\,(g'_1,\ldots,g'_n)\in G^{(n)}$ and $\{s(g_1),g_1,\ldots,g_n\}=\{s(g'_1),g_1',\ldots,g_n'\},$ then 
\begin{equation}
    \Omega(g_1,\ldots,g_n)=\Omega(g_1',\ldots,g_n').
    \end{equation}\
We use the notation
\begin{align*}
    &S^nG^*=\text{ subspace of completely symmetric n-cochains}\;,
    \\& S^n_0G^*=\text{ subspace of normalized completely symmetric n-cochains}\;.
\end{align*}
\end{definition}
$\,$\\The point of the degeneracy condition in the previous definition is that an $n$-cochain $\Omega$ satisfying it naturally defines an $S_{k+1}$-invariant $k$-cochain for any $k\le n,$ eg. on $(g_1,\ldots, g_k),$ $\Omega$ equals $\Omega(g_1,\ldots,g_k,s(g_1),\ldots, s(g_1)),$ where we have repeated $s(g_1)$ $(n-k)$ times. 
\begin{exmp}
Let $\mu$ be a locally finite Borel measure on $\mathbb{R}^2.$ For $x,y,z\in\mathbb{R}^2$, let $C_{(x,y,z)}$ be the relative interior of the convex hull containing $x,y,z$ (ie. the interior of the convex hull as a manifold with boundary). Then 
\begin{equation}
    \Omega(x,y,z)=\mu(C_{(x,y,z)})
    \end{equation}
is a completely symmetric $2$-cochain on $\textup{Pair}\,\mathbb{R}^2$ (which isn't necessarily continuous). 
\end{exmp}
Most of the paper will emphasize the following cochains:
\begin{definition}\label{even}
Let $\mathcal{A}^n_0G^*$ denote smooth, normalized $n$-cochains which are invariant under $A_{n+1}$ (even permutations).
\end{definition}
\begin{lemma}
We have the following decomposition:
\begin{equation}
    \mathcal{A}^n_0G^*=S^{n}_0G^*\oplus\Lambda^{n} G^*\;.
    \end{equation}    
\end{lemma}
$\,$
\\\\The direct sum 
\begin{equation}
    \mathcal{A}^nG^*=S^{n}G^*\oplus\Lambda^{n} G^*
    \end{equation}
will also play an important role. These are the cochains which are invariant under even permutations $A_{n+1}$ and are well-defined on degeneracies. These are, in great generality, the things we can integrate over an $n$-dimensional manifold $X,$ when $G=\textup{Pair}\,\textup{X}.$ The first summand is like the space of signed measures (normalized non-differentiable ones are like densities), they don't need an orientation to be integrated. The second summand is like the space of top forms, they require an orientation to be integrated. The ones in lower degree can be integrated over continuous maps via the pullback. We'll discuss this in detail later.
\begin{exmp}\label{first}
Let $f:[0,1]\to\mathbb{R}$ be a smooth function and let $F'=f.$ We have some $1$-cochains on $\textup{Pair}\,[0,1]$ given by 
\begin{align}
   & (x,y)\mapsto f(x)(y-x)\,,
   \\&(x,y)\mapsto f\Big(\frac{x+y}{2}\Big)(y-x)\,,
   \\&(x,y)\mapsto f(y)(y-x)\,,
   \\&(x,y)\mapsto F(y)-F(x)\,.
\end{align}
These are all in $\mathcal{A}^n_0\textup{Pair}\,[0,1]^*$ and the second and fourth cochains are in $\Lambda^n\textup{Pair}\,[0,1]^*$. The first three correspond to the left hand, midpoint and right hand Riemann sums, respectively.
\\\\Let's compare them by fixing $x$ and computing their Taylor expansions up to order $2,$ centered at $y=x.$ Stated differently, let's compute their second order jets along the source fibers. We get, respectively,
\begin{align}
&f(x)(y-x)\,,
\\&f(x)(y-x)+\frac{1}{2}f'(x)(y-x)^2\,,
\\& f(x)(y-x)+f'(x)(y-x)^2\,,
\\& f(x)(y-x)+\frac{1}{2}f'(x)(y-x)^2\,.
\end{align}
We suggestively rewrite these Taylor expansions as 
\begin{align}
&f(x)\,dx\,,
\\&f(x)\,dx+\frac{1}{2}f'(x)\,dx^2\,,
\\& f(x)\,dx+f'(x)\,dx^2\,,
\\& f(x)\,dx+\frac{1}{2}f'(x)\,dx^2\,.
\end{align}These all agree to first order but to second order only the second and fourth cochains agree. As a consequence, these cochains are all equivalent from the perspective of smooth maps, but from the perspective of Brownian paths only the second and fourth are equivalent. That is, if we pull back these cochains by a smooth map and integrate we get the same value for all of them, but if we pull them back by Brownian paths we only get the same values for the second and fourth cochains. We will discuss this more in \cref{brown}. 
\\\\We haven't defined $VE_{\bullet}$ yet, but up to first order these Taylor expansions will be obtained by applying $VE_0$ to the cochains, and the Taylor expansions up to second order will be obtained by applying $VE_1.$
\end{exmp}
\subsubsection{Cochains on Lie Algebroids}
Some things here might be described in a slightly unconventional way in order to display relationships with cochains on groupoids.
\\\\Associated to every vector bundle $\pi:V\to M$ is a groupoid with $s=t=\pi,$ where the multiplication structure is given by fiberwise addition. Therefore, we naturally have multilinear antisymmetric (symmetric) cochains on vector bundles. Since every Lie algebroid has an underlying vector bundle,  we get the following:
\begin{definition}\label{multil}
We denote smooth, pointwise multilinear maps $\mathfrak{g}^{\oplus n}\to\mathbb{R}$ by $C^n(\mathfrak{g}).$
\end{definition}
Here, given a Lie algebroid $\mathfrak{g}\to X,$ $\mathfrak{g}^{\oplus n}\to X$ is the $n$-fold direct sum of the underlying vector bundle. We denote by $\mathfrak{g}^{\otimes n}\to X$ the $n$-fold tensor product. 
\begin{definition}
A map $\omega\in C^n(\mathfrak{g})$ is an $n$-form if it is antisymmetric under the action of $S_n.$ We denote these by $\Gamma(\Lambda^n\mathfrak{g}^*).$
\end{definition}
\begin{exmp}
Let $\mathfrak{g}=TX.$ Then an $n$-form on $\mathfrak{g}$ is an $n$-form on $X.$
\end{exmp}
Therefore, unless stated otherwise (as in the following definition), an $n$-form is antisymmetric by default.
Due to example \ref{first}, we will want to consider functions on $\mathfrak{g}^{\oplus n}$ which pointwise are polynomial functions. Recall:
\begin{definition}
Let $V$ be a vector space. A homogeneous map of degree $i$ on $V$ is a map $f:V\to\mathbb{R}$ which satisfies $f(\lambda v)=\lambda^if(v),$ for all $\lambda\in\mathbb{R}$ and $v\in V;$ it is positive homogeneous if this condition is satisfied for $\lambda > 0.$ A polynomial on $V$ is a smooth function $f:V\to\mathbb{R}$ which is a finite sum of homogeneous functions with degrees $i\in\{0,1,2,\ldots\}.$
\end{definition}
\begin{definition}\label{polyg}
We denote by $\Gamma(\mathcal{A}^n_{k}\mathfrak{g}^*)$ the space of smooth maps $\mathfrak{g}^{\oplus n}\to\mathbb{R}$ which, pointwise, are polynomials of degree $k$ and are invariant under $A_{n+1}.$
\end{definition}
\begin{exmp}
In example \ref{first}, $dx^2$ is a pointwise homogeneous map of degree $2$ on $T\mathbb{R}$; it is also completely symmetric. These higher order terms can't be discarded when integrating over Brownian paths.
\end{exmp}
\begin{exmp}
Let $g:TX\otimes TX\to\mathbb{R}$ be a Riemannian metric on $X.$ Then the map 
\begin{equation}
    TX\to\mathbb{R},\;\xi\mapsto\sqrt{g(\xi,\xi)}
    \end{equation}
is completely symmetric and positive homogeneous of degree $1.$ Of course, we can integrate it over curves to compute their arclength. 
\end{exmp}
\begin{exmp}
A $2$-form on a manifold $X$ is in $\Gamma(\mathcal{A}^2_{2}TX^*).$
\end{exmp}
An interesting implication is that given a vector space $V,$ there is a ``secret" action of $S_{n+1}$ on the $n$-fold product $V^{\oplus n}.$ The non-permutation part of the action is given by 
    \begin{equation}
        (v_1,v_2,\ldots,v_n)\mapsto (-v_1,-v_1+v_2,\ldots,-v_1+v_n)\,.
    \end{equation}
Note that, this action doesn't descend to the tensor product. One reason we don't hear about this action is because of the following:
    \begin{proposition}\label{same}
    Let $V$ be a finite dimensional vector space and let $f:V^{\oplus n}\to\mathbb{R}$ be a smooth map that is positive homogeneous of degree $1$ in each component. Then $f$ is multilinear, and it is invariant under $A_{n+1}$ if and only if it is antisymmetric under $S_{n+1}$ if and only if it is antisymmetric under $S_n.$
    \end{proposition}
\subsubsection{*Densities}
Lastly, we are going to define densities. Specialized to $TX,$ these are the objects which one integrates over $X,$ where $n=\text{dim}(X).$ Often we integrate $n$-forms, but to do so requires a choice of orientation, and an $n$-form together with an orientation determines a density on $TX$ (of course, we can integrate measures, but these aren't defined using $TX$). These have a standard definition on $TX.$
\\\\First, we can think of an orientation of an $n$-dimensional vector space $V$ as a nontrivial map $\mathcal{O}:V^{\oplus n}\to \mathbb{R}$ which satisfies 
\begin{equation}
    \mathcal{O}(Av_1,\ldots,Av_n)=\text{sgn}(\det{A})\mathcal{O}(v_1,\ldots,v_n)\,.
\end{equation}
Equivalently, we can think of it as a nontrivial completely antisymmetric map which is positive homogeneous of degree $0$ in each component. Given a trivial vector bundle $\mathbb{R}^n\to X,$ a local orientation is given by an orientation of $\mathbb{R}^n.$
\begin{definition}
A smooth density on a rank $n$ Lie algebroid $\mathfrak{g}\to X$ is a map $\mu:\mathfrak{g}^{\oplus n}\to\mathbb{R}$ such that, given a local trivialization and a local orientation $\mathcal{O},$ the product $\mathcal{O}\mu$ is a smooth $n$-form (where we restrict $\mu$ to the local trivialization).
\end{definition}
Pointwise, a density satisfies $\mu(A\xi_1,\ldots,A\xi_n)=|\det{A}|\mu(\xi_1,\ldots,\xi_n)$ and we only need local orientations to talk about smoothness. In the same vein as \ref{same}, we have:
\begin{proposition}
A map $\mathfrak{g}^{\oplus n}\to\mathbb{R}$ is a density if and only if it is completely symmetric and positive homogeneous of degree $1$ in each component. It is smooth if locally it is smooth after multiplying by an orientation $\mathcal{O}.$
\end{proposition}
\begin{exmp}
Let $dx$ be the standard $1$-form on $\mathbb{R},$ then $|dx|$ is a density. With the standard orientation on $T\mathbb{R},$ $\mathcal{O}|dx|=dx.$
\end{exmp}
\subsection{*Simplicial Maps and the Groupoid Differential}\label{diff}
We complete the construction of the nerve by writing out the face and degeneracy maps, and we explain a relationship between symmetrization and the groupoid differential.
\\\\Let $G\rightrightarrows X$ be a Lie groupoid. There are $n+2$ face maps 
\begin{equation}
    \delta_0,\ldots,\delta_{n+1}:G^{(n+1)}\to G^{(n)},\,n\ge 0.
    \end{equation}
For $n=0$ we have that $\delta_0=t,\delta_1=s.$ For $n\ge 1$ and for $k\ge 1$ we have that $\delta_k$ is the projection which drops the $kth$ arrow, and 
\begin{equation}
\delta_0(g_1,\ldots,g_{n+1})=(g_1^{-1}g_2,\ldots,g^{-1}_1g_{n+1})\,.
\end{equation}
There are also $n+1$ degeneracy maps 
\begin{equation}
    \sigma_0,\ldots,\sigma_n:G^{(n)}\to G^{(n+1)},\,n\ge 0.
    \end{equation}
For $n=0$ we have $\sigma_0(x)=\text{id}(x).$ For $n\ge 1$ we have $\sigma_0(g_1,\ldots,g_n)=(s(g_1),g_1,\ldots,g_n),$ 
\begin{equation}
    \sigma_k(g_1,\ldots,g_n)=(g_1,\ldots,g_{k-1},s(g_1),g_{k},\ldots,g_n)\;,\;\;k\ge 1\;. \end{equation}
The following definition is standard:
\begin{definition}\label{diffg}
For all $n\ge 0,$ there is a differential $\delta^*:C^n(G)\to C^{n+1}(G)$ which is given by the alternating sum of the pullbacks of the face maps, ie.
\begin{equation}
    \delta^*\Omega=\frac{1}{(n+1)!}\sum_{i=0}^n(-1)^i\delta_i^*\Omega\;.
\end{equation}
\end{definition}
This differential satisfies $\delta^{*2}=0$ and we can take its cohomology.
\begin{proposition}
The differential $\delta^*$ restricts to a differential $\delta^*:\Lambda^{n}G^*\to\Lambda^{n+1}G^*$ and the inclusion $\Lambda^{n}\,G^*\xhookrightarrow{} C^n(G)$ is a quasi-isomorphism.
\end{proposition}
\subsubsection{*Conceptual Explanation}
There is a conceptual way of thinking about the differential. For each $k\ge 0$ there are natural maps $p_k:G^{(n+k)}\to G^{(n)}$ given by the projection dropping the last $k$ arrows — if $n=0$ we map onto the common source. We get a map $\Lambda^nG^*\to \Lambda^{n+k}G^*$ by pulling back via $p_k$ and completely antisymmetrizing the resulting cochain, ie.
\begin{equation}
    \Omega\mapsto \text{Alt}_{}[p_k^*\Omega\big]\;,
\end{equation}
where $\text{Alt}$ is antisymmetrization map with respect to $S_{n+k+1}.$\footnote{Since $\Omega$ is completely antisymmetric, it doesn't matter which arrows one decides to drop, we just chose the last ones for brevity.} For $k\ge 2$ this map is zero since $p_k^*\Omega$ only depends on the first $n$ factors and for all $(g_1,\ldots,g_n,g_{n+1},\ldots,g_{n+k})$ there is a transposition fixing the first $n$ factors, so it is case that 
\begin{equation}
    \text{Alt}_{}[p_k^*\Omega\big]=-\text{Alt}[p_k^*\Omega\big]\;.
\end{equation}
On the other hand,
\begin{equation}
    \text{Alt}_{}[p_2^*\Omega\big]= \text{Alt}_{}\big[p_1^*\text{Alt}_{}[p_1^*\Omega\big]\big]
\end{equation}
since every permutation of $\{0,1,\ldots,n+2\}$ can be factored as a permutation fixing $n+2$ followed by a transposition. Thus $\text{Alt}\,p_1^*$ is a differential; we have that 
\begin{equation}
  \text{Alt}\,p_1^*=\delta^*\;.
\end{equation}
Therefore, $\delta^*$ restricts to a differential on completely antisymmetric cochains. There is a corresponding map on (normalized) completely symmetric cochains, but it's of no interest to us because it creates too many ``atoms" to be integrated (ie. nonvanishing points on the identity).
\subsection{Local Lie Groupids}\label{locg}
We will find it useful to make use of local Lie groupoids (see eg. \cite{cabr}), which is what we get when we restrict the space of arrows of a groupoid to a neighborhood of the identity bisection. As a result, not all arrows which are composable in the groupoid are composable in the local groupoid, however ones which are close enough to the identity bisection are. 
\begin{definition}
Let $G\rightrightarrows X$ be a Lie groupoid. Let $U$ be a neighborhood of $X\subset G^{(1)}$ which is closed under inversion. We call this a local Lie groupoid and denote it by $G_{\textup{loc}}\rightrightarrows X.$   
\end{definition}
For the constructions we want to make there isn't much of a difference between a local Lie groupoid and a groupoid. We can think of a morphism of local Lie groupoids $f:H_{\textup{loc}}\to G_{\textup{loc}}$ as one which satisfies $f(h_1\cdot h_2)=f(h_1)\cdot f(h_2)$ whenever the composition on the left side makes sense in $H_{\textup{loc}}.$ More precisely:
\begin{definition}
We have a simplicial manifold $G_{\textup{loc}}^{\bullet}$ given by the largest sub-simplicial manifold of $G^{\bullet}$ such that  $G_{\textup{loc}}^{0}=X,\,G_{\textup{loc}}^{1}=U.$
\end{definition}
In general, when we speak of the local groupoid we will assume that the fibers of $s:U\to X$ are contractible.
\\\\
Using this definition of the nerve, all of the structures defined in the previous sections naturally carry over to the local groupoid. Furthermore, we will be leaving $U$ implicit when we talk about local groupoids and thus we won't clearly distinguish between different local Lie groupoids (since all we really care about is the germ near the identity bisection).
\begin{exmp}\label{locp}
An open cover $\{V_i\}_i$ of $X$ determines a local groupoid $\textup{Pair}\,\textup{X}_{\textup{loc}}\rightrightarrows X,$ where
\begin{equation}
   \textup{Pair}^{(n)}\,X_{\textup{loc}}= \{(x_0,\ldots,x_n)\in X^{n+1}:\, \{x_0,\ldots,x_n\}\subset V_i\;\textup{for some }\, V_i\}\;.
\end{equation}
\end{exmp}
\section{The van Est Map}\label{ve1}
In this section we will state a generalization of the van Est map — we define it in a simpler way than the original was defined (see \cref{vanest}). This is a review of work done by the author in \cite{Lackman3}. In \cref{ve2} we will generalize it further to a graded version. For now, we will define the van Est map to be a differentiation map from normalized cochains on Lie groupoids to multilinear maps on Lie algebroids: \begin{equation}
    VE_0:C^n_0(G)\to C^n(\mathfrak{g})\;.
    \end{equation}
Recall that a vector $\xi\in\mathfrak{g}$ at a point $x\in X$ is a vector tangent to the source fiber of $G$ at $x.$ Now, given an $n$-cochain $\Omega$ and a point $x\in X,$ we can restrict 
\begin{equation}
    \Omega:\underbrace{G\sideset{_s}{_{s}}{\mathop{\times}} \cdots\sideset{_s}{_{s}}{\mathop{\times}} G}_{n \text{ times}}\to\mathbb{R}
\end{equation}to a map
\begin{equation}
    \Omega_x:\underbrace{s^{-1}(x){\mathop{\times}}  \cdots{\mathop{\times}} s^{-1}(x)}_{n \text{ times}}\to\mathbb{R}\,,
\end{equation}so it makes sense to independently differentiate $\Omega$ in each of the $n$ components.
\begin{definition}
Let $G\rightrightarrows X$ be a Lie groupoid and $\mathfrak{g}\to X$ its corresponding Lie algebroid. For each $n\ge 1$ we define a map 
\begin{equation}
    VE{_0}:C^n_0(G)\to C^n(\mathfrak{g})\,,\;\; \Omega\mapsto VE_0(\Omega)
\end{equation}
as follows: for $\xi_1,\ldots,\xi_n\in\mathfrak{g}$ over $x\in X,$ we let \begin{equation}\label{same}
    VE{_0}(\Omega)(\xi_1,\ldots,\xi_n)=\xi_n\cdots\xi_1\Omega_x\;,
    \end{equation}
where $\xi_i$ differentiates in the $ith$ component of $\Omega_x.$ If $\Omega$ is a $0$-cochain we define $VE{_0}(\Omega)=\Omega.$
\end{definition}
\begin{remark}
This definition makes sense on $G_{\textup{loc}}$ as well, without modification.
\end{remark}
\begin{lemma}
The image of $\mathcal{A}^n_0G^*$ under $VE_0$ is contained in $\Gamma(\Lambda^n\mathfrak{g}^*)$ (in addition, it surjects onto this space).
\end{lemma}
This is due to the fact that smooth cochains in $S^n_0G^*$ get mapped to $0$ under $VE_0.$
\\\\Work done in \cite{Lackman3} proves that $VE{_0}$ is related to $VE$ via antisymmetrization, hence up to a factor of $n!$ they agree on normalized antisymmetric cochains.
\begin{exmp}
$VE_0$ applied to the cochains $(2.1.6)-(2.1.9)$ in example \ref{first} gives the first order terms of $(2.1.14)-(2.1.17).$ We will soon extend $VE_0$ so that it includes the second order terms.
\end{exmp}
\begin{exmp}
Let $\Omega=\Omega(x_0,y_0,x_1,y_1,x_2,y_2)$ be a $2$-cochain on $\textup{Pair}\,\mathbb{R}^2.$ Then
\begin{equation}
    VE_0(\Omega)(\partial_x\vert_{(x_0,y_0)},\partial_y\vert_{(x_0,y_0)})=\partial_{y_2}\partial_{x_1}\Omega(x_0,y_0,\cdot,\cdot,\cdot,\cdot)\vert_{(x _1,y_1)=(x_2,y_2)=(x_0,y_0)}\;.
\end{equation}
\end{exmp}
\begin{exmp}
Riemannian metrics correspond to normalized symmetric (but not completely symmetric)\footnote{This can be resolved by thinking of Riemannian metrics as exclusively defined on the tensor product of the tangent bundle with itself, not on the direct sum.} $2$-cochains $\Omega,$ which are positive in the sense that $\Omega(g,g)\ge 0\,.$ For example, take the standard metric $g\in C^2(T\mathbb{R}^2)$ on two dimensional Euclidean space:
\begin{equation}
    g((a,b),(c,d))=(a,b)\cdot (c,d)\;.
    \end{equation}
We get it by applying $VE_0$ to $\Omega\in C^2(\textup{Pair}\,\mathbb{R}^2),$
\begin{equation}
    \Omega(x_0,y_0,x_1,y_1,x_2,y_2)=(x_1-x_0,y_1-y_0)\cdot(x_2-x_0,y_2-y_0)\,.
\end{equation}
\end{exmp}
\begin{exmp}
Consider the $n$-cochain on $\textup{Pair}\,\mathbb{R}^n$ given by
\begin{equation}
    \Omega(x_0,x_1,\ldots,x_n)=\frac{1}{n!}\det{[x_1-x_0,\ldots,x_n-x_0]}\;,
\end{equation}
where $x_i=(x_i^1,\ldots,x_i^n).$ Then $VE_0(\Omega)=dx^1_0\wedge\cdots\wedge dx^n_0$ (see \cref{warning} for our convention for wedge products).
\end{exmp}
As a consequence of normalization of cochains, in any coordinate system on $s^{-1}(x),$ the $n$th-order Taylor expansions of $\Omega_x$ is determined by $VE_0(\Omega)$ at $x.$ In other words, $VE_0$ determines the $n$-jet of $\Omega_x.$ In particular, the following lemma is needed to justify our construction of the integral in \cref{int}:
\begin{lemma}\label{asymptoticsn}
Suppose that the source fibers of $G\rightrightarrows X$ are $n$-dimensional and that $\Omega\in \mathcal{A}^n_0G^*.$ Let $y=(y^1,\ldots,y^n)$ be coordinates on $s^{-1}(x)$ in a neighborhood of $x,$ whose coordinate is $x=(x^1,\ldots,x^n).$ This determines a coordinate system on the $n$-fold product $s^{-1}(x)\times\cdots\times s^{-1}(x)$ in a neighborhood of $x,$ written $(y_1,\ldots,y_n),$ where $y_i=(y_i^1,\ldots,y_i^n),\,1\le i\le n.$ 
\\\\Then the $n$th-order Taylor expansion of $\Omega_x$ centered at $x,$ evaluated at $(y_1,\ldots,y_n),$ is equal to
\begin{equation}
VE(\Omega)(\partial_{y^1},\ldots,\partial_{y^n})\det{[y_1-x,\ldots,y_n-x]}\;.
\end{equation}
 Here, $\partial_{y^{j}}$ is the coordinate vector field at the point $x.$
\end{lemma}
\begin{remark}
We can use the previous result to motivate an extension of $VE_0$ so that densities are in its image (essentially, by replacing the determinant with its absolute value when using completely symmetric cochains).
\end{remark}
Moving on, there is a standard definition of a differential on a Lie algebroid. However, keeping in the spirit of emphasizing the groupoid over the algebroid:
\begin{definition}\label{diffa}
Let $G\rightrightarrows X$ be a Lie groupoid. There is a differential 
\begin{equation}
    d:\Gamma(\Lambda^{\bullet}\mathfrak{g}^*)\to \Gamma(\Lambda^{\bullet+1}\mathfrak{g}^*)
\end{equation}
that is uniquely defined by the property that on $\Lambda^{\bullet}G^*$
\begin{equation}
    VE_0\,\delta^*=d\,VE_0\;.
\end{equation}
\end{definition}
\begin{exmp}
If $\mathfrak{g}=TX$ then $d$ is the exterior derivative.
\end{exmp}
We will now state a simple version of the van Est isomorphism theorem, proved by Crainic \cite{Crainic}. We will state a more general one in \cref{funm}, which will be needed when discussing Chern-Simons with finite gauge group.
\begin{theorem}(see \cite{Crainic}, \cite{Meinrenken})
Assume the source fibers of $G\rightrightarrows X$ are $n$-connected. Then $VE_0$
(restricted to completely antisymmetric cochains) defines an isomorphism between the cohomology of $G$ and the cohomology of $\mathfrak{g}$ up to degree $n,$ and is injective in degree $n+1.$\footnote{In these papers, it's stated without assuming complete antisymmetry because conventionally one just antisymmetrizes after differentiating.}
\end{theorem}
In the context of $TX,$ $VE_0$ defines an isomorphism between the cohomology of $\text{Pair}\,X_{\textup{loc}}$ and the de Rham cohomology of $X.$ See \cref{diff}.
\section{Generalized Riemann Sums}
The construction of the functional integral we would like to make suggests a generalized notion of Riemann sums on manifolds.\footnote{There are other things called generalized Riemann sums and generalized Riemann integrals in the literature, but these are different from what we define.} Recall that Riemann sums on manifolds are typically defined only locally, not globally. We will recall the global definition, defined originally by the author in \cite{weinstein1}. Before doing so, we will motivate the definition, but first we recall some basic algebraic topology:
\begin{definition}\label{trian}
A smooth triangulation of a compact manifold (with boundary) $M$ is given by a simplicial complex $\Delta_M,$ together with a homeomorphism $|\Delta_M|\to M$ from the geometric realization to $M,$ such that the restriction of this map to each simplex is a smooth embedding. We denote the set of smooth triangulations of $M$ by $\mathcal{T}_M.$
\end{definition}
For some of the constructions we do, the actual map $|\Delta_M|\to M$ doesn't matter except for on vertices, and thus we will often leave this map implicit and refer to $\Delta_M$ as a triangulation. Associated to every triangulation $\Delta_M$ is a simplicial set, obtained by picking a total ordering of the vertices. If $M$ is oriented we should choose an ordering which is compatible with the orientation. We won't distinguish $\Delta_M$ from this simplicial set.
\\\\We will assume standard facts about triangulations of compact manifolds (with boundary), eg. they always exist and triangulations of the boundary can be extended to the entire manifold. See Munkres \cite{munk} for details.
\\\\Let $M$ be a manifold and consider an $n$-cochain $\Omega$ on $\textup{Pair}\,\textup{M}.$ Suppose we have an embedding of the standard $n$-simplex $|\Delta^n|$ inside $M$ and we want to assign a number to it — we could do this by choosing an ordering of the vertices $(v_0,\ldots,v_{n})$ and assigning it the value $\Omega(v_0,\ldots,v_n).$ However, in order for this quantity to be independent of the ordering we need $\Omega$ to be completely symmetric. If in addition we want something like local absolute continuity with respect to the Lebesgue measure, then $\Omega$ should be normalized as well. 
\\\\Assuming $\Omega$ is completely symmetric, we can assign a value to a given triangulation $\Delta_M$ by forming the generalized Riemann sum
\begin{equation}\label{summ}
   \sum_{\Delta \in\Delta_M }\Omega(\Delta)\;.
\end{equation}
Here we sum over all simplices, making sure to sum over degenerate simplices only once (in the case that $\Omega$ isn't normalized). We can ``integrate" $\Omega$ by taking a direct limit over triangulations.
\\\\To be precise, we take the limit in the sense of nets, over an equivalence class of triangulations:
\begin{definition}\label{equiv}
Two triangulations are equivalent if they have a common linear subdivision\footnote{That is, subdivisions of a simplex $|\Delta|\xhookrightarrow{} |\Delta_M|$ are induced by linear subdivisions of $|\Delta|.$} (eg. barycentric subdivisions of a given triangulation). The triangulations in an equivalence class form a directed set ordered by linear subdivision.\end{definition}
In great generality these limits are independent of the equivalence class chosen, as we will see in the next section.
\\\\On the other hand, if $\Omega$ is not completely symmetric but is only invariant under $A_{n+1}$ and well-defined on degeneracies, then in order to assign a value to an $n$-simplex we need a choice of ordering of the vertices, up to even permutation. An orientation of $M$ induces an ordering (up to even permutation) of the vertices of each $n$-simplex of the triangulation of $M,$\footnote{For example, by picking a vertex and looking at the orientation of the tangent vectors at that vertex corresponding to the $1$-dimensional faces.} and therefore we can still use \ref{summ}. 
\begin{definition}\label{deff}
Let $\Omega\in \mathcal{A}^n \textup{Pair}\,\textup{M}^*_{\textup{loc}},$\footnote{Recall the equality $\mathcal{A}^n G^*=S^n G^*\oplus\Lambda^n G^*.$} where $M$ is compact $n$-dimensional (with boundary, and oriented if necessary). We define 
\begin{equation}
\int_M\Omega=\lim_{\Delta_M\in \mathcal{T}_M}   \sum_{\Delta \in\Delta_M }\Omega(\Delta)
\end{equation}
if the limit of \ref{summ} exists over each equivalence class of triangulations and is independent of the equivalence class.
\end{definition}
\begin{remark}
If a cochain $\Omega$ is to define a finite measure, it should (in particular) have at most countably many atoms, ie. points on the identity bisection where $\Omega$ doesn't vanish.
\end{remark}
\begin{exmp}
Let $f,g:[0,1]\to\mathbb{R}.$ Let $\Omega$ be the $1$-cochain on $\textup{Pair}\,[0,1]$ given by $s^*f(t^*g-s^*g).$ Then for a triangulation $0=x_0<x_1<\cdots<x_n=1$ we have that
\begin{equation}
    \sum_{\Delta}\Omega(\Delta)=\sum_{i=0}^{n-1}f(x_i)(g(x_{i+1})-g(x_i))\;.
\end{equation}
This converges to the Riemannn-Stieltjes integral $\int_0^1 f\,dg\,.$
\end{exmp}
\begin{exmp}\label{eulerc}\textbf{(Euler Characteristic)}
Let $M$ be an $n$-dimensional compact manifold (with boundary) and let $\Omega\in S^n\textup{Pair}\,\textup{M}^*$ be defined by 
\begin{equation}
\Omega(x_0,\ldots,x_n)=(-1)^{|\{x_0,\ldots,x_n\}|+1}\;.
\end{equation}
Then $\int_M\Omega=\chi(M)\,.$
\end{exmp}
$\,$\\
Of course, this $\Omega$ doesn't define a \textit{countably} additive signed measure. In the next section we will discuss Gauss-Bonnet as well.
\\
\begin{exmp}
Let $M$ be an $n$-dimensional manifold and let $m\in M.$ We have an $n$-cochain $\Omega$ on $\textup{Pair}\,\textup{M}$ given by $\Omega(m_0,\ldots, m_n)=1$ if $m_i=m$ for all $0\le i\le n,$ and $0$ otherwise. Let $f:M\to\mathbb{R}$ and let $s$ denote the common source map $\textup{Pair}^{(n)}\,M\to M.$ We have that, for any triangulation $\Delta_M$ having the point $m$ as a vertex, 
\begin{equation}
    \sum_{\Delta\in\Delta_M}s^*f\,\Omega(\Delta)=f(m)\;.
\end{equation}
Therefore, the limit is $f(m)\,.$ Thus, $\Omega$ is an $n$-cochain representing the Dirac measure concentrated at $m.$
\end{exmp}
If $\Omega$ is not just invariant under even permutations, but actually completely antisymmetric, we get the following result:
\begin{lemma}(Triangulated Stokes' Theorem)\label{stok}
Let $M$ be a compact oriented $(n+1)$-manifold with boundary and let $\Omega$ be a completely antisymmetric $n$-cochain. Let $\Delta_M$ be a triangulation of $M$ with induced triangulation $\Delta_{\partial M}$ of the boundary. Then
   \begin{equation}
    \sum_{\Delta \in\Delta_{\partial M }}\Omega(\Delta^n)=\sum_{\Delta \in\Delta_{M }}\delta^* \Omega(\Delta^{n+1})\,,
\end{equation} 
where $\delta^*$ is the groupoid differential on $\textup{Pair}\,\textup{M}.$
\end{lemma}
\begin{proof}
Since $\Omega$ is completely antisymmetric, this follows from the fact that an $n$-dimensional face in the interior appears as a face twice, with opposite orientations.
\end{proof}
\section{Fundamental Theorem of Calculus on Manifolds and the Main Corollary}\label{funm}
Let $\omega$ be an $n$-form on an oriented $n$-manifold $M$ (with boundary). Given the discussion in the previous section, we can assign a Riemann sum to $\omega$ by triangulating $M$ and assigning to $\omega$ a normalized cochain that is invariant under even permutations, denoted $\Omega.$ The defining property of $\Omega$ is that $VE{_0}(\Omega)=\omega.$ 
\\\\Of course, Stokes' theorem and the Poincar\'{e} lemma are usually regarded as the analogues of the fundamental theorem of calculus on manifolds (which the following section doesn't use prior knowledge of). However, we will state a further generalization, one that is more natural in the context we are working in. In addition, the language makes sense in the non-smooth category and can be used to understand the fundamental theorem of calculus for the Riemann-Stieltjes integral. We include a part $0.$
\begin{theorem}(Part 0)\label{int}
Let $M$ be an oriented compact $n$-dimensional manifold (with boundary), let $\omega$ be an $n$-form on $M$ and let 
\begin{equation}
    \Omega\in \mathcal{A}_0^n\textup{Pair}\,\textup{M}^*_{\textup{loc}}
\end{equation}
satisfy $VE{_0}(\Omega)=\omega.$ Then
\begin{equation}
   \lim\limits_{\Delta_M\in\mathcal{T}_M} \sum_{\Delta \in\Delta_M }\Omega(\Delta)=\int_M \omega\;,
\end{equation}
where the limit is taken in the sense of \ref{deff}.
\end{theorem}
\begin{proof}
This was proved by the author in \cite{Lackman3} and uses \cref{asymptoticsn}. We will repeat the $1$-dimensional case here, the higher dimensional case is similar. 
\\\\The result is true if and only if it's true locally, so we can assume we are on an interval $[0,1].$ Let $f\,dx$ be a $1$-form and let $\Omega$ be a normalized cochain on $\textup{Pair}\,[0,1]$ such that $VE_0(\Omega)=f\,dx.$ We then have that the cochain
\begin{equation}
    (x,y)\mapsto \Omega(x,y)-f(x)(y-x)
\end{equation}
maps to $0$ under $VE_0,$ and for a triangulation $0=x_0<x_1,\ldots<x_n=1$ we can write
\begin{equation}
    \sum_{i=0}^{n-1}\Omega(x_i,x_{i+1})=\sum_{i=0}^{n-1}f(x_i)(x_{i+1}-x_i)+\sum_{i=0}^{n-1}\big[\Omega(x_i,x_{i+1}) -f(x_i)(x_{i+1}-x_i)\big]\;.
\end{equation}
The first term on the right converges to $\int_0^1f\,dx,$ so we just need to show that the second term converges to $0.$ From Taylor's theorem and the fact that $\Omega(x,x)=0,$ we have that for $y> x$
\begin{equation}
    \Omega(x,y)-f(x)(y-x)=(y-x)\,\frac{\partial}{\partial y'} \big[\Omega(x,y')-f(x)(y'-x)\big]\vert_{y'=\xi_{x,y}}
\end{equation}
for some $\xi_{x,y}\in (x,y).$ The condition on $VE_0$ implies that
\begin{equation}
   \lim_{y'\to 0} \frac{\partial}{\partial y'} \big[\Omega(x,y')-f(x)(y'-x)\big]=0
\end{equation}
uniformly, and the result follows.
\end{proof}
We've stated this result for differential forms but it can also be stated for densities.
\begin{theorem}(Part 1)\label{main2} If in addition $\partial M=\emptyset,$ the symmetric part of $\Omega$ is zero and $\Omega$ is closed under $\delta^*,$ then the generalized Riemann sum approximation is exact for any triangulation, ie.
\begin{equation}
\sum_{\Delta\in\Delta_M}\Omega(\Delta)=\int_M \omega\;.
\end{equation}More generally, if $\partial M\ne\emptyset$ then we have a map (see \ref{rela})
\begin{equation}
    H^n\textup{Pair}\,(M,\partial M)_{\textup{loc}}\to \mathbb{R}
    \end{equation}
given by the \textup{relative} Riemann sum
\begin{equation}\label{rel}
(\Omega_M,\Omega_{\partial M})\mapsto \sum_{\Delta\in\Delta_M}\Omega_M(\Delta)\;-\sum_{\Delta\in\Delta_{\partial M}}\Omega_{\partial M}(\Delta)\;,
\end{equation}
and the right side is equal to 
\begin{equation}
\int_M VE_0(\Omega_M)-\int_{\partial M}VE_0(\Omega_{\partial M})\;.
\end{equation}
\end{theorem}
\begin{remark}\label{rela}
Here, $H^n\textup{Pair}\,(M,\partial M)_{\textup{loc}}\to \mathbb{R}$ is the relative groupoid cohomology, defined analogously to relative de Rham cohomology: a $k$-cochain is a pair
\begin{equation}
    (\Omega_M,\Omega_{\partial M})\in \Lambda^k \textup{Pair}\,\textup{M}^*_{ loc}\oplus\Lambda^{k-1}\textup{Pair}\,\partial M^*_{ loc}\;,
    \end{equation}
and the differential is given by 
\begin{equation}
    (\Omega_M,\Omega_{\partial M})\mapsto (\delta^*\Omega_M,i^*\Omega_M-\delta^*\Omega_{\partial M})\;,
    \end{equation}
    where $i:\textup{Pair}^{(k)}\partial M_{ loc}\to \textup{Pair}^{(k)}M_{ loc}$ is the map induced by the inclusion $\partial M\xhookrightarrow{} M.$
\end{remark}
\begin{proof}
This follows by picking any two triangulations of $M,$ pulling back $(\Omega_M,\Omega_{\partial M})$ to 
\begin{equation}
    \textup{Pair}\,\textup{M}_{\textup{loc}}\times\textup{Pair}\,[0,1]
    \end{equation}
via the projection onto $\textup{Pair}\,\textup{M}_{\textup{loc}}$ (the pullback of $(\Omega_M,\Omega_{\partial M})$ will still be closed), extending the two triangulations of $M$ to a triangulation of 
\begin{equation}
    M\times [0,1],
    \end{equation}
applying \cref{stok} to deduce that \ref{rel} is independent of the triangulation and finally applying \cref{int}. 
\end{proof}
Part 2 is most generally stated for a Lie groupoid $G\rightrightarrows X.$ It is the van Est isomorphism theorem for $G$-modules and is stated and proved in \cite{Lackman}. A stronger statement can be found in \cite{Lackman2}. We will state a simple version of it here:
\begin{theorem}(Part 2)\label{part2}
Let $A$ be an abelian Lie group. Then $VE_0$ induces an isomorphism 
\begin{equation}
    H^*(G_{\textup{loc}},A)\cong H^*(\mathfrak{g},A)\;.
\end{equation}
\end{theorem}
$\,$\\In the above, the left side is the sheaf cohomology of $A$-valued cochains on the nerve of $G_{\textup{loc}}\,.$ The right side is the Lie algebroid cohomology valued in the trivial $A$-bundle, which is obtained using the Lie algebroid version of the exterior covariant derivative. Explicitly, it is the cohomology of  
\begin{equation}
   \mathcal{O}(A)\xrightarrow[]{\textup{dlog}}\mathcal{O}(\mathfrak{g}^*\otimes\mathfrak{a})\xrightarrow[]{d} \mathcal{O}(\Lambda^2\mathfrak{g}^*\otimes\mathfrak{a})\to\cdots\;,
\end{equation}
where $\mathcal{O}(A)$ is the sheaf of $A$-valued functions on the base, and $\mathcal{O}(\Lambda^k \mathfrak{g}^*\otimes\mathfrak{a})$ is the sheaf of $k$-forms valued in the Lie algebra $\mathfrak{a}$ of $A.$ 
$\,$\\\\We will give a construction which is similar to the one found in \cite{cabr}.
\begin{proof}
We will sketch the proof for $A=\mathbb{R},$ the proof for general $A$ is done similarly using C\v{e}ch cocycles.  We generalize the construction of the antiderivative in the fundamental theorem of calculus. First, we choose an identification of 
\begin{equation}
    \begin{tikzcd}\label{ident}
G_{\text{loc}} \arrow[d] \arrow[r, shift right=7] & \mathfrak{g} \arrow[d] \\
X                                    & X                     
\end{tikzcd}
\end{equation}
which is the identity on $X,$ and for which the derivative restricts to the identity map on $\mathfrak{g}\subset TG.$\footnote{We're using the natural identification of $V$ with the tangent space at the origin for a vector space $V.$ Here, $V$ is a fiber of $\mathfrak{g}.$ This is called a tubular structure in \cite{cabr}.} 
\\\\Let $\omega$ be a closed $n$-form on $\mathfrak{g}.$ For $g_1,\ldots,g_n\in G_{\textup{loc}}$ with source $x\in X,$ let $C_{(g_1,\ldots,g_n)}$ be the convex hull of $x,g_1,\ldots g_n,$ defined using \ref{ident}.\footnote{That is, convex hulls make sense in a vector space.} This space is naturally oriented by the vectors $(g_1-x,\ldots,g_n-x),$ if they are linearly independent. The following is a completely antisymmetric $n$-cocycle which maps to $\omega$ under $VE_0:$
\begin{equation}
    \Omega(g_1,\ldots, g_n)=\int_{C_{(g_1,\ldots,g_n)}}\omega\;,
\end{equation}
where we identify $\omega$ with its left translation to $G_{\textup{loc}}\,.$ 
\end{proof}
Part 0 of this theorem is the motivation for the groupoid approach to functional integration we describe (or really, the groupoid approach to functional integration suggested part 0). Parts 1 and 2 are motivated by topological field theories, eg. the Poisson sigma model and Chern-Simons for finite groups. In particular, part 1 tells us that for a topological sigma model, we only need to take a limit over triangulations of the  boundary of the domain to compute the action of a morphism — if the boundary is empty we don't need to take any limit to compute it.
\begin{remark}
Of course, Part 1 is related to Lefschetz duality: $(\Omega_M,\Omega_{\partial M})$ pulls back to a relative simplicial cocycle and the relative generalized Riemann sum is the cap product with the fundamental class of $M.$ If we assume the cap product is used, we don't need to assume the symmetric part is zero. Furthermore, we can make the stronger statement that the perfect pairings of $H^*\textup{Pair}\,(M,\partial M)_{\textup{loc}}\,,$ induced by singular cohomology and de Rham cohomology, are equal.
\end{remark}
$\,$\\The following is the main result that justifies our approach to functional integration (and holds in a bit more generality). First, note that by Lie's second theorem a Lie algebroid morphism $\mathbf{x}:TM\to\mathfrak{g}$ integrates uniquely to a morphism $\mathbf{X}:\textup{Pair}\,\textup{M}_{\textup{loc}}\to G_{\textup{loc}}.$
\begin{corollary}\label{funco}
Let $M$ be $n$-dimensional and let $\omega$ be an $n$-form on $\mathfrak{g}.$ Let $\Omega\in \mathcal{A}_0^n G^*$ satisfy $VE_0(\Omega)=\omega.$ Then for $\mathbf{x}:TM\to\mathfrak{g},$
\begin{equation}
    \lim_{\Delta_M\in\mathcal{T}_M}\sum_{\Delta\in\Delta_M}\mathbf{X}^*\Omega(\Delta)=\int_M \mathbf{x}^*\omega\;,
\end{equation}
where $\mathbf{X}:\textup{Pair}\,\textup{M}_{\textup{loc}}\to G_{\textup{loc}}$ integrates $\mathbf{x}.$ If $\Omega$ is a cocycle and $\partial M=\emptyset,$ then the generalized Riemann sum of $\mathbf{x}^*\Omega$ is exact.
\end{corollary}
\begin{proof}
    This follows from \cref{int}, \cref{main2} and fact that $VE_0,\,\delta^*$ are natural with respect to pullbacks by morphisms.
\end{proof}
\begin{remark}
With regards to \ref{main2}, \ref{funco}, if $\Omega$ is a cocycle and $\partial M\ne\emptyset,$ then the generalized Riemann of $\Omega$ won't depend on the triangulation of the interior, just on the triangulation of the boundary.
\end{remark}
\begin{exmp}
Let $\Omega\in \mathcal{A}_{0}^1\textup{Pair}\,[0,1]^*$ and let $f:[0,1]\to\mathbb{R}$ be smooth and let $0=x_0<x_1<\cdots<x_n=1$ be a triangulation of $[0,1].$
Then
\begin{equation}
    \sum_{\Delta\in \Delta_{[0,1]}}\Omega(\Delta)=\sum_{i=0}^{n-1}\Omega(x_i,x_{i+1})\;.
\end{equation}
If we let 
\begin{equation}
    \Omega(x,y)=f(x)(y-x),\;f(y)(y-x)
    \end{equation}
we get the left and right Riemann sums, respectively. If we let $F'=f$ and $\Omega(x,y)=F(y)-F(x)$ we get $F(1)-F(0)$ (which is exact because this $\Omega$ is a cocycle). We could also consider 
\begin{equation}
    \Omega(x,y)=f(xe^{y-x})\sin{(y-x)}+f(xy)(y-x)^5\;.
    \end{equation}
These all map to $f\,dx$ under $VE_0,$ so according to \cref{int} the limits of the generalized Riemann sums all equal $\int_0^1 f\,dx\,.$
\end{exmp}
\begin{exmp}\label{funex}
Parts 1 and 2 generalize the fundamental theorem of calculus. We get 
\begin{equation}
    f(b)-f(a)=\int_a^b df
    \end{equation}
by applying part 1 to the relative $1$-cocycle $(\Omega_{[a,b]},\Omega_{\{a,b\}})$ with 
\begin{equation}
    \Omega_{[a,b]}(x,y)=f(y)-f(x)\,,\;\;\Omega_{\{a,b\}}=0.
    \end{equation}
For part 2, the map
\begin{equation}
  (a,x)\mapsto\int_a^x df
\end{equation}
is the cocycle $\Omega$ we defined in the proof of part 2. Therefore, $VE_0(\Omega)=df\,,$ which says that the exterior derivative of $x\mapsto \int_a^x df$ is $df$ for each $a.$
    \end{exmp}
    \begin{exmp}\label{highf}
This is a higher dimensional example of the previous example. Let $M$ be $n$-dimensional and let $(\Omega,0)$ be an  $n$-cocycle on $\textup{Pair}\,(M,\partial M)$ (which implies $\Omega$ must vanish on the boundary). Then 
\begin{equation}
    \sum_{\Delta\in\Delta_M}\Omega(\Delta)=\int_M VE_0(\Omega)\;.
\end{equation}
    \end{exmp}
\begin{exmp}\label{stokex}
$\,$\\\begin{itemize}
    \item We can get Stokes' theorem and the Poincar\'{e} lemma: Stokes' theorem 
\begin{equation}
    \int_{\partial M}\omega=\int_M d\omega
\end{equation}
is a special case of part 1, which is obtained by applying it to the relative cocycle
\begin{equation}
    (\Omega_M,\Omega_{\partial M})=\delta^*(\Omega,0)
\end{equation}
where $VE_0(\Omega)=\omega\,;$ we just need to know that $VE_0$ is surjective on completely antisymmetric cochains, which is easy to show using coordinates and a partition of unity.
\item For the Poincar\'{e} lemma on $\mathbb{R}^m,$ we get a primitive for $\omega$ by trivializing the cocycle $\Omega$ in the proof of part 2. Explicitly, we define  
\begin{equation}
 \Omega_{0}\in \Lambda^{n-1}\textup{Pair}\,\mathbb{R}^{m*}\,,\;\;\;\Omega_{0}(x_1,\ldots,x_n)=\int_{C(0,x_1,\ldots,x_n)}\omega\;,
 \end{equation}
where $C(0,x_1,\ldots,x_n)$ is the convex hull of $0,x_1,\ldots, x_n\in\mathbb{R}^m.$ We then have that \begin{equation}
    d\,VE_0(\Omega_{0})=\omega\;,
    \end{equation}
this follows from the fact that $d\,VE_0=VE_0\,\delta^*.$ 
\end{itemize}
$\,$\\Note that, a short computation shows that for any $M$ and any $m\in M,$ the map 
\begin{equation}
    \mathcal{A}^{n}\textup{Pair}\,\textup{M}^*\to \mathcal{A}^{n-1}\textup{Pair}\,\textup{M}^*\;,\;\;\;\Omega_m(m_1,\ldots,m_n)=\Omega(m,m_1,\ldots,m_n)
\end{equation}
trivializes any cocycle $\Omega\,.$ An abstract reason for the existence of a primitive is due to the fact that the pair groupoid is Morita equivalent to a point. 
\end{exmp}
\subsubsection{Gauss-Bonnet Example}
We discussed the Euler characteristic in example \ref{eulerc}. We will now discuss the Gauss-Bonnet theorem.
\begin{exmp}
Let $(M,g)$ be an oriented Riemannanian surface (with boundary). We get a degree $2$-cocycle $(\Omega_M,\Omega_{\partial M})$ as follows:
let $\Omega_M\in\Lambda^2\textup{Pair}\,\textup{M}_{\textup{loc}}^*$ be given by
\begin{equation}
    \Omega(x_0,x_1,x_2)=\pm(\textup{sum of internal angles of the corresponding geodesic triangle}-\pi)\;,\footnote{This is how the curvature is defined on a simplicial complex in Regge \cite{regge}.}
\end{equation}
where the sign is chosen according to whether the vectors determined by $(x_0,x_1),\,(x_0,x_2)$ are oriented or not.
\\\\As for $\Omega_{\partial M}\in\Lambda^1\textup{Pair}\,\partial M_{\textup{loc}}^*,$ it is given by
\begin{equation}
    \Omega_{\partial M}(x_0,x_1)=\mp\textup{(sum of internal angles of the corresponding semicircle})\;,
\end{equation}
where the semicircle is determined by the arc on the boundary and the geodesic connecting $x_0, x_1,$ and the sign is chosen according to whether the vector on $\partial M$ determined by $(x_0,x_1)$ is oriented or not (it's minus if oriented)
\\\\The standard counting arguments (eg. at each interior vertex the adjacent angles add up o $2\pi$) show that the relative Riemann sum is equal to \begin{equation}
    2\pi(\textup{number of vertices}-\textup{number of edges}+\textup{number of faces)}\,
\end{equation}and \cref{main2} shows that this is exactly equal to 
\begin{equation}
    \int_M VE_0(\Omega_M)-\int_{\partial M}VE_0(\Omega_{\partial M})\;.
    \end{equation}
   Of course, this is the Gauss-Bonnet theorem.
\end{exmp}
\begin{remark}
We obtained the fundamental theorem of calculus as a special case of \ref{highf}, not of \ref{funex}.
This was intentional because the fundamental theorem of calculus is only a special case of  Stokes' theorem by convention; one must treat the $1$-dimensional case separately. We are viewing the fundamental theorem of calculus as being about primitives with respect to $VE_0,$ not with respect to $d.$
\end{remark}
\section{What Can We Integrate? What Are the Observables?}
In great generality, the things that we integrate over an $n$-dimensional manifold $X$ can be described by germs of completely symmetric $n$-cochains defined on a neighborhood of $X$ (more precisely, a jet of such a cochain). In fact, we can construct such a cochain from any Borel measure: choose a triangulation $|\Delta_X|$ of $X,$ and consider the space $U\subset \textup{Pair}^{(n)}\,X$ consisting of $n$-tuples of points in $X$ which lie in the same simplex in $|\Delta_X|.$ On $U$ we can define a completely symmetric $n$-cochain by assigning to an $n$-tuple the measure of the interior\footnote{We mean the interior of the convex hull as a manifold with boundary.} of the convex hull containing those points. We can extend this cochain by zero to a small open set containing $U.$\
\begin{exmp}(see \cite{horst} for a more general criteria)
Let $\mu$ be a finite Borel measure on $[0,1]$ and let $\mu$ be the $1$-cochain constructed in the preceding paragraph. Then for any continuous function $f,$
\begin{equation}
    \int_{[0,1]} f\,d\mu=\lim_{\Delta_{[0,1]}\in \mathcal{T}_{[0,1]}}\sum_{\Delta\in\Delta_{[0,1]}}(s^*f\mu)(\Delta)\;.
\end{equation}
\end{exmp}
The cochain we constructed won't be continuous, but if the measure is given by a density we can integrate it to a continuous, completely symmetric cochain on the pair groupoid. We frequently obtain measures by using top forms instead of densities, but in order to do so we need a choice of orientation, and such a choice determines a density.
\\\\On the other hand, instead of just integrating over manifolds we frequently want to integrate over maps $f:Y\to X,$ eg. 
\begin{equation}
    \int_Y f^*\omega\,,
    \end{equation}
as we do when computing functional integrals. More generally, we can pull back Lie algebroid forms by morphisms $f:TY\to\mathfrak{g}$ and integrate them too. At the level of cochains on the groupoid, the objects we pull back and integrate are
\begin{equation}
 \mathcal{A}^nG^*\cong S^nG^*\oplus \Lambda^n G^*\;.
\end{equation}
Again, we only need to consider the germs of such cochains near $X\subset G^{(n)},$ since only the limit of \ref{summ} is important.
\\\\Let us remark some advantages completely antisymmetric cochains have over completely symmetric ones: they come with a product, a differential and they satisfy Stokes' theorem exactly (\ref{stok}). Also, the cochains of interest in $\Lambda^{\bullet} G^*$ are more likely to be smooth than the cochains of interest in $S^{\bullet} G^*,$ eg. consider $\Omega(x,y)=y-x$ versus $\Omega(x,y)=|y-x|.$ We want to include the latter in the domain of $VE_0$ so that we can say it differentiates to $|dx|.$
\\\\In more generality, there are things we can integrate over maps $f:Y\to X$ that aren't pulled back from $\text{Pair}(X),$ eg. given a function $V:X\to\mathbb{R}$ and a density $\mu$ on $Y,$ we can integrate $(f^*V)\mu\,.$ Generally speaking then, given an $n$-dimensional manifold $Y,$ the objects we want to integrate over a \textit{continuous} Lie algebroid morphism\footnote{We think of continuous (but not differentiable) Lie algebroid morphisms as continuous morphisms between their integrating groupoids. Recall that the Wiener measure is supported on non-differentiable paths.} $f:TY\to \mathfrak{g}$ are given by germs near the identity bisections of
\begin{equation}\label{suma}
\big[S^n \textup{Pair}\,Y^*_{\textup{loc}}\otimes S^0 G^*_{\textup{loc}}\big]\oplus \big[S^0 \textup{Pair}\,Y^*_{\textup{loc}}\otimes S^n G^*_{\textup{loc}}\big]\oplus \bigoplus_{i+j=n}\Lambda^i\textup{Pair}\,Y^*_{\textup{loc}}\otimes\Lambda^jG^*_{\textup{loc}}\;.
\end{equation}
That is, we use a map $f:\textup{Pair}\,Y_{\textup{loc}}\to G_{\textup{loc}}$ to pull back the cochain data on $G_{\textup{loc}}$ and then we take the product with the cochain data on $\textup{Pair}\,Y_{\textup{loc}},$ see definition \ref{wedge}. We then get approximations to observables by forming Riemann sums. We get approximations to more general observables by choosing $k$ such observables and inserting them into a function $F:\mathbb{R}^k\to\mathbb{R}.$ 
\\\\Even more generally, we can combine forms/densities and integrate over submanifolds. One way of approximating these observables is by choosing triangulations which restrict to triangulations of these closed submanifolds. This is what is done in path integrals when computing correlation functions, ie. if we want to compute the expectation value of $\mathbf{x}(t_1)\mathbf{x}(t_2)$ then we look at triangulations that have $t_1, t_2$ as vertices.\footnote{Which, according to the ordering by linear subdivision, is equivalent to looking at triangulations greater than some triangulation with $t_1,t_2$ as vertices.}
\\\\In the next section we use this framework to describe several stochastic integrals appearing in Brownian motion.
\subsection{Cochains Needed for Brownian Motion}\label{brown}
This discussion motivates the definition of the graded version of the van Est map. We assume standard facts about Brownian motion regarding stochastic integrals. See eg. \cite{peter} for details. 
\\\\
The significance of using germs of cochains to integrate is related to the fact that for a smooth function $f:\mathbb{R}\to\mathbb{R}$ and a partition $0=t_0\le t_1\le\ldots\le t_n=1$ of $[0,1],$ the random variables $C([0,1],\mathbb{R})\to\mathbb{R}$ given by
\begin{align}\label{it}
   &\mathbf{x}\mapsto \sum_{i=0}^{n-1}f(\mathbf{x}(t_i))(\mathbf{x}(t_{i+1})-\mathbf{x}(t_i))\;,
   \\&\mathbf{x}\mapsto\sum_{i=0}^{n-1}\frac{f(\mathbf{x}(t_i))+f(\mathbf{x}(t_{i+1}))}{2}(\mathbf{x}(t_{i+1})-\mathbf{x}(t_i))
\end{align}
converge to different random variables in $L^2(C([0,1],\mathbb{R}),\mu_W)$\footnote{$\mu_W$ is the Wiener measure.} as we take the limit over triangulations. On top we have the It\^{o} integral and on the bottom we have the Stratonovich integral. 
\\\\To explain this from the perspective of $VE{_0},$ consider the $1$-cochains on $\textup{Pair}\,\mathbb{R}$ given by 
\begin{equation}\label{itos}
\Omega_I(x,y)=f(x,y)(y-x)\;,\;\;\;\Omega_S(x,y)=\frac{f(x)+f(y)}{2}(y-x)\;.
\end{equation}
The It\^{o} and Stratonovich integrals sum over $\mathbf{x}^*\Omega_I,\,\mathbf{x}^*\Omega_S,$ respectively. If $\mathbf{x}$ has sufficient regularity, eg. is smooth, then pointwise the two sums converge to the same quantity, which would be the integral of $\mathbf{x}^*f\,dx\,.$ The problem is that while 
\begin{equation}
    VE{_0}(\Omega_I-\Omega_S)=0\,,
\end{equation}
we don't necessarily have  
\begin{equation}\label{hol}
    VE{_0}(\mathbf{x}^*\Omega_I-\mathbf{x}^*\Omega_S)=0\,.
\end{equation}
A simple computation shows that \ref{hol} holds if
\begin{equation}
   \lim\limits_{t_1\to t_0} \frac{(\mathbf{x}(t_1)-\mathbf{x}(t_0))^2}{t_1-t_0}=0\,,
\end{equation}
which is true if $\mathbf{x}$ is H\"{o}lder continuous with exponent $>1/2.$ However, Brownian paths with this degree of regularity have measure zero — Brownian paths have nearly as good a regularity as possible without being H\"{o}lder continuous with exponent $1/2,$ eg. they are H\"{older} continuous on a set of full measure for any exponent less than $1/2.$
\\\\For a normalized $1$-cochain $\Omega,$ $VE{_0}$ only remembers the first order Taylor expansion of $\Omega_x$ at $x,$ for each $x\in\mathbb{R}.$ This is enough if we are integrating over smooth paths, however for Brownian paths we need to remember the second order Taylor expansion of $\Omega_x,$ which is due to the following standard fact: 
\begin{proposition}
Let $\Omega(x,y)=(y-x)^2$ and for $\mathbf{x}\in C([0,1],\mathbb{R})$ consider the $1$-cochain on $\textup{Pair}[0,1]$ given by 
\begin{equation}
    \mathbf{x}^*\Omega(t,t')=(\mathbf{x}(t')-\mathbf{x}(t))^2\,.
    \end{equation}
For a partition $0=t_0\le t_1\le\ldots\le t_n=1$ of $[0,1],$ consider the random variable 
\begin{equation}
   \mathbf{x}\mapsto  \sum_{i=0}^{n-1} \mathbf{x}^*\Omega(t_i,t_{i+1})\,.
\end{equation}  
This converges to $1$ in $L^2(C([0,1],\mathbb{R}),\mu_W)$ as we take the limit over triangulations.
\end{proposition}
Therefore, in order to integrate normalized $1$-cochains $\Omega=\Omega(x,y)$ on $\text{Pair}(\mathbb{R})$ over Brownian paths, we need to remember the second order Taylor expansions, ie. map
\begin{equation}
   T\mathbb{R}\to\mathbb{R}\,,\;\;\;\lambda\partial_y\mapsto \lambda\partial_y\Omega(x,x)+\frac{\lambda^2}{2}\partial_y^2\Omega(x,x)\,.
\end{equation}
This means that the objects on $T\mathbb{R}$ we want to pullback and integrate over $[0,1]$ aren't of the form $A(x)\,dx,$ they are of the form 
\begin{equation}
    A(x)\,dx+B(x)dx^2\,;
\end{equation}
The first term is a degree one homogeneous map $T\mathbb{R}\to\mathbb{R}$ (ie. a one form\footnote{Any degree one homogeneous map $f:V\to\mathbb{R}$ on a vector space $V$ is automatically linear if it is smooth at the origin.}) and second term is a degree two homogeneous map $T\mathbb{R}\to\mathbb{R}.$ 
\begin{definition}\label{pointh}
For a manifold $X,$ let $\mathcal{H}_i(TX)$ denote smooth maps $TX\to\mathbb{R}$ which are pointwise homogeneous of degree $i.$
\end{definition}
\begin{definition}
For a $1$-cochain $\Omega$ on $\textup{Pair}\,\mathbb{R},$ let 
\begin{equation}
    VE_1(\Omega)\in \mathcal{H}_0(T\mathbb{R})\oplus\mathcal{H}_1(T\mathbb{R})\oplus \mathcal{H}_2(T\mathbb{R})
    \end{equation}
be equal to 
\begin{equation}
    \Omega(x,x)+\partial_2\Omega(x,x)\,dx+\frac{1}{2}\partial_2^2\Omega(x,x)\,dx^2\,.
\end{equation}
\end{definition}
\begin{lemma}
Let $\Omega_1,\Omega_2\in S_0^1\textup{Pair}\,\mathbb{R}$ be smooth, such that $VE_1(\Omega_1)=VE_1(\Omega_2).$ For a partition $0=t_0\le t_1\le\ldots\le t_n=1$ of $[0,1],$ consider the random variables
\begin{align}\label{it}
   &\mathbf{x}\mapsto \sum_{i=0}^{n-1}\mathbf{x}^*\Omega_1(t_i,t_{i+1})\;,
   \\&\mathbf{x}\mapsto\sum_{i=0}^{n-1}\mathbf{x}^*\Omega_2(t_i,t_{i+1})\;.
\end{align}
These random variables converge to the same random variable in $L^2(C([0,1],\mathbb{R}),\mu_W)$ as we take the limit over triangulations..
\end{lemma}
We will be defining $VE_{k}$ so that on a groupoid $G\rightrightarrows X,$ $VE_k(\Omega)$ computes the $(n+k)$th order Taylor expansion of an $n$-cochain $\Omega_x,$ for each $x\in X$ (or rather, we will use jets to do this).
\\\\Returning to the It\^{o} and Stratonovich integral and the cochains defined in \ref{itos}, we have that
\begin{equation}
VE_1(\Omega_I)=f(x)\,dx\,,\;\;\; VE_1(\Omega_S)=f(x)\,dx+\frac{1}{2}f'(x)\,dx^2\,.
\end{equation}
As we can see, the two cochains agree to order one, but not to order two, which is why the random variables in \ref{it} converge to different random variables. There is a natural embedding 
\begin{equation}
    \mathcal{H}_1(T\mathbb{R})\xhookrightarrow{}  \mathcal{H}_1(T\mathbb{R})\oplus  \mathcal{H}_2(T\mathbb{R})
\end{equation}
given by letting the degree $2$ term be zero, and from this perspective the It\^{o} integral is correct.
\\\\On the other hand, due to \ref{stok}, the Stratonovich integral satisfies the fundamental theorem of calculus since its cochain is completely antisymmetric, whereas the It\^{o} integral does not. In fact, one can check that the integral is unaffected by the choice of completely antisymmetric cochain as long as it has the correct first order term. 
\subsection{*Completely Antisymmetric Cochains in Physics}\label{physics}
Completely antisymmmetric cochains are implicitly used in physics as the midpoint prescription for Riemann sums is used for magnetic potentials in the path integral (see eg. \cite{gav}). This is done in order to recover the correct solution to Schr\"{o}dinger's equation.\footnote{Also classically, the defining property of the magnetic potential is that it satisfies Stokes' theorem with the magnetic field, so for this reason complete antisymmetry should be expected.} That is, one can say that when there is a magnetic potential $A\,dx$ in the classical action, what we are really integrating is
\begin{equation}\label{antip}
    S[\mathbf{x}]=\int_0^1 \frac{1}{2}\dot{\mathbf{x}}^2+V(\mathbf{x})\,dt+A(\mathbf{x})\dot{\mathbf{x}}\,dt+\frac{1}{2}A'(\mathbf{x})\dot{\mathbf{x}}^2 \,dt^2\;.
\end{equation}
For a smooth path $\mathbf{x},$ this is indistinguishable from
\begin{equation}\label{antipp}
    S[\mathbf{x}]=\int_0^1 \frac{1}{2}\dot{\mathbf{x}}^2+V(\mathbf{x})\,dt+A(\mathbf{x})\dot{\mathbf{x}}\,dt\;.
\end{equation}
To get from \ref{antipp} to \ref{antip} we can perform the complete antisymmetrization at the level of pointwise polynomials on $T\mathbb{R},$ ie. of $A(x)\,dx.$ This is a higher order antisymmetrization map. That is, previously we considered the nontrivial element of $S_2$ to act on $A(x)\,dx$ by 
\begin{equation}
    A(x)\,dx\mapsto -A(x)\,dx\,,
\end{equation}
but this is only the right action up to degree $1,$ which isn't good enough for Brownian paths. Consider the action of $S_2$ on 
\begin{equation}
   \prod_{i=0}^{\infty}\mathcal{H}_i(T\mathbb{R}) 
\end{equation} 
defined on homogeneous elements by
 \begin{equation}\label{high}
     A(x)\,dx^n\mapsto (-1)^n\sum_{k=0}^{\infty}\frac{1}{k!}A^{(k)}(x)\,dx^{n+k}\,.
 \end{equation}
Using this action the antisymmetrization of $A(x)\,dx$ is, up to order two,
\begin{equation}
    A(x)\,dx+\frac{1}{2}A'(x)\,dx^2\,.
\end{equation}
This leads to a higher order de Rham differential. For a smooth function $f:\mathbb{R}\to\mathbb{R}$ it agrees with the jet prolongation (see eg. \cite{saund}) and is given by
\begin{equation}
    f\xmapsto{}f'(x)\,dx+\frac{1}{2}f''(x)\,dx^2\;.
\end{equation}
Assuming the framework of integration we have been describing using cochains, this version of the de Rham differential should satisfy the fundamental theorem of calculus on $C([0,1],\mathbb{R})$ with respect to any measure which is concentrated on H\"{o}lder continuous functions with exponent greater than $1/3.$
\subsection{*Topological Quantum  Field Theories}\label{tqfts}
A \textit{relatively} simple class of Lie algebroid-valued quantum field theories are ones with homotopy invariant actions. TQFTs are some of the most well-studied QFTs, eg. Chern-Simons. For the definition of homotopy of Lie algebroid morphisms and their relevance to the Poisson sigma model, see  Bojowald–Kotov–Strobl \cite{bojo}, also Bonechi–Zabzine \cite{zab}.
\begin{definition}
Let $Y$ be an $n$-dimensional manifold (with boundary). A Lie algebroid-valued quantum field theory with paths $\mathbf{x}:TY\to\mathfrak{g}$ is topological if the action only depends on the homotopy class of $\mathbf{x},$ relative to the boundary. 
\end{definition}
$\,$\\In this case, the observables will turn out to only depend on homotopy classes of maps, relative to the boundary (we discuss this in the last section). In particular, if the action is of the form 
\begin{equation}
    S[\mathbf{x}]=\int_Y \mathbf{x}^*\omega
\end{equation}
for some cocycle $\omega\in\Lambda^n\mathfrak{g}^*,$ then the quantum field theory is topological. However, we will see more general topological actions when discussing Chern-Simons with finite gauge group.
\begin{exmp}
Feynman and Wiener's path integrals are not topological.
\end{exmp}
\begin{exmp}
    Let $(M,\omega)$ be a symplectic manifold. For a surface $Y$ we have an action given by
    \begin{equation}
        S[\mathbf{x}]=\int_Y \mathbf{x}^*\omega\;.
    \end{equation}
This theory is topological and is a special case of the Poisson sigma model. We will discuss this example more later.
\end{exmp}
\section{The Graded van Est Map}\label{ve2}
Motivated by the discussion in the previous sections, we define a version of the van Est map; it is graded in the sense that its image is a graded vector space if we have a Riemannian metric along the source fibers . We will be using jets, for a textbook account see Saunders \cite{saund}.
\begin{definition}
Let $G\rightrightarrows X$ be a (local) Lie groupoid and let $C^n(X)$ be germs of $n$-cochains\footnote{We only really need to require them to be Boreal measurable.} near the identity bisection (ie. we identify functions defined on $C^n(G)$ if they agree on some neighborhood of the identity bisection). For $k\ge -n$ let $I_k^n(X)$ be the subspace of $C^n(X)$ consisting of those $n$-cochains $\Omega$ such that for each $x\in X,$ the derivatives of order $\le n+k$ of $\Omega_x$ vanish at $x$ (if this is true in one coordinate system it is true in all coordinate systems). We define 
\begin{equation}
    J^n_k(X)=C^n(X)/I^n_k(X)\;.
\end{equation}
\end{definition}
\begin{definition}\label{jetv}
Let $G\rightrightarrows X$ be a Lie groupoid. For an $n$-cochain $\Omega$ and for $k\ge -n,$ let
\begin{equation}
    VE_k^n:C^n(G)\to J^n_k(X)
\end{equation}
be the natural quotient map. If $k<-n$ define $VE_k=0.$
\end{definition}
Note that, in practice we suppress the upper index and just write $VE_k.$ 
\\\\There is a natural identification of $C^n(\mathfrak{g})$ with  the subspace of normalized cochains in $J^n_0(X),$ and under this identification the graded van Est map, with $k=0,$ agrees with the previously defined $VE_0\,.$
\\\\Note that, \ref{jetv} makes sense even for cochains in $C^n(G)$ that aren't smooth. We will discuss this further when we discuss the Riemann-Stieltjes integral on manifolds in the next section.
\begin{remark}
If we have a Riemannian metric along the source fibers of $G\rightrightarrows X$ then we can identify the smooth subspace inside $J^n_k(X)$ with sections of
\begin{equation}
\mathbb{R}\oplus \mathfrak{g}^{*\oplus n}\oplus \text{Sym}^2\mathfrak{g}^{*\oplus n}\oplus\cdots\oplus \text{Sym}^{n+k}\mathfrak{g}^{*\oplus n}\;,\footnote{Since we can use a metric to split the short exact sequence of jet bundles $0\to\text{Sym}^n T^*M\to J^n(M)\to J^{n-1}(M)\to 0.$ See eg. \cite{vakil}. Also see covariant Taylor expansions in \cite{avra}.}
\end{equation}
where for a vector space $V,$ $\text{Sym}^iV^*$ are linear maps on $V^{\otimes i}$ which are invariant under the permutation action of $S_i.$ If $P\in S^i(V)$ then we get a homogeneous function of degree $i$ on $V$ given by $v\mapsto P(v,\ldots,v),$ for $v\in V.$\\\\
Therefore, when we have such a metric, we can view $VE_k$ as being a map
\begin{equation}
    VE_k:C^n(G)\to  \mathcal{H}^n_0(\mathfrak{g})\oplus\mathcal{H}_1^n(\mathfrak{g})\oplus\cdots\oplus \mathcal{H}^n_{n+k}(\mathfrak{g})\,,
\end{equation}
where $\mathcal{H}_i^n(\mathfrak{g})$ are degree $i$ homogeneous functions on $\mathfrak{g}^{\oplus n}.$ When working with normalized cochains the image is concentrated in degree $\ge n.$ It would be interesting to describe the complete action of $A_{n+1}$ on the right hand side, see \ref{high}.
\end{remark}
For examples, see example \ref{first} and \cref{brown}.
\section{**Riemann-Stieltjes Integrals on Manifolds}
This framework suggests a simple generalization of the Riemann-Stieltjes integral to manifolds. Note that, at least grammatically, a lot of what we are doing makes sense for piecewise linear manifolds (and even pseudomanifolds), and one can extend the theorems proved in \cref{funm}.
\begin{definition}\label{RSm}
Let $M$ be $n$-dimensional and oriented, and let $\Omega\in\Lambda^{n-1}\textup{Pair}\,\textup{M}^*.$ Then for $f:M\to\mathbb{R}$ we can define the following limit (if it exists):\footnote{We may need to fix an equivalence class of triangulations to do this; at the moment it's not clear to what extent this definition is independent of equivalence classes.}
\begin{equation}\label{rs}
\int_M fd\Omega=\lim_{\Delta_M\in\mathcal{T}_M}\sum_{\Delta\in\Delta_M}(s^*f(\Delta))\delta^*\Omega(\Delta)\;.
\end{equation}
\end{definition}
Before addressing existence of the limit, we make the following definition:
\begin{definition}\label{tota}
Let $\Omega\in \mathcal{A}^n\textup{Pair}\,\textup{M}^*_{\textup{loc}}$\footnote{For more general groupoids, one way of making sense of a cochain's total variation being zero is by pulling it back to pair groupoids.} where $M$ is $n$-dimensional (and oriented if necessary). We define its total variation to be\footnote{Depending on the cochain, one may have to fix an equivalence class of triangulations in order to talk about its total variation.}
\begin{equation}
\limsup_{\Delta_M\in\mathcal{T}_M}\sum_{\Delta\in\Delta_M}|\Omega|\;.
\end{equation}
\end{definition}
\begin{exmp}
Let $f:[0,1]\to\mathbb{R}$ be a continuous function. Then its total variation in the traditional sense is the total variation of $\delta^*f.$
\end{exmp}
For
\begin{equation}
    \int_a^b f\,dg
\end{equation}
to exist it is enough that $f$ is continuous and that $g$ has bounded variation, which according to our \cref{tota} means that $\delta^*g$ has bounded variation. The analogue holds on manifolds:
\begin{proposition}
Suppose that $f$ is continuous and that $\delta^*\Omega$ has bounded variation. Then \ref{rs} exists (for the given equivalence class of triangulations which we take the total variation with respect to).
\end{proposition}
\begin{proof}
The proof is an adaptation of the standard proof for the case of an interval. The result is true if and only if it's true on each top dimensional face of our geometric triangulation, so we will assume $M=|\Delta^n|.$ First, we need to show that if $|\Delta_M|\le |\Delta_M'|,$ then
\begin{equation}\label{inc}
    \sum_{\Delta\in\Delta_M}|\delta^*\Omega|\le\sum_{\Delta\in \Delta_M'}|\delta^*\Omega|\;.
\end{equation}
This follows from \cref{stok} (the triangulated Stokes' theorem) and the triangle inequality. Next, for $(x_0,\ldots,x_n)\in \textup{Pair}^{(n)}\,M,$ define $V_{(x_0,\ldots,x_n)}(\delta^*\Omega)$ to be the total variation of $\delta^*\Omega$ over linear subdivisions of the convex hull of $x_0,\ldots,x_n\in M.$ We then get a completely symmetric $n$-cochain $V_{\bullet}(\delta^*\Omega)$ given by
\begin{equation}
    V_{\bullet}(\delta^*\Omega)(x_0,\ldots,x_n)=V_{(x_0,\ldots,x_n)}(\delta^*\Omega)\;.
\end{equation}
 From \ref{inc}, it follows that $V_{\bullet}(\delta^*\Omega)$ is additive in the sense that if $M$ is subdivided by the top-dimensional faces $\Delta_1,\ldots,\Delta_k,$ then
 \begin{equation}
      V_{\bullet}(\delta^*\Omega)(M)=\sum_{i=1}^k V_{\bullet}(\delta^*\Omega)(\Delta_i)\;.
 \end{equation}
 The same is true for $V_{\bullet}(\delta^*\Omega)-\delta^*f.$
We can write the right side of \ref{rs} as
\begin{equation}
    \lim_{\Delta_M}\sum_{\Delta\in\Delta_M}(s^*f)V_{\Delta}(\delta^*\Omega)\;-\sum_{\Delta\in\Delta_M}(s^*f)(V_{\Delta}(\delta^*\Omega)-\delta^*\Omega(\Delta))\;.
\end{equation}
Therefore, to show convergence it is enough to show convergence of both terms. This follows by observing that the following terms go to $0$ as we take the limit over triangulations:
\begin{align}
    &\sum_{\Delta\in\Delta_M}(\sup_{|\Delta|}{f}-\inf_{|\Delta|}{f})V_{\Delta}(\delta^*\Omega)\;,
    \\& \sum_{\Delta\in\Delta_M}(\sup_{|\Delta|}{f}-\inf_{|\Delta|}{f})(V_{\Delta}(\delta^*\Omega)-\delta^*\Omega(\Delta))\;,
\end{align}
where the infimum and supremum are taken over all points in $|\Delta|.$
\end{proof}
More generally, this result holds if we replace $\delta^*\Omega$ with any completely antisymmetric cocycle since cocycles on the local pair groupoid are locally trivial (ie. on the full subgroupoid over a small enough open set in the base).
\\\\We can construct the associated (signed) measure as a direct limit in the categorical sense, as we did for the Lebesgue measure in \cref{lebe}. At least on the interval, all regular Borel measures are obtained this way via the Riesz representation theorem, ie. for a $0$-cochain $\Omega$ with $\delta^*\Omega\ge 0$ on oriented $1$-simplicies.
\begin{definition}
Let $M$ be an $n$-dimensional (oriented) manifold with triangulation $|\Delta_M|,$ and let $\Omega,\Omega'\in \mathcal{A}^n\textup{Pair}\,\textup{M}^*_{\textup{loc}}.$ We say that $\Omega\sim\Omega'$ if the total variation of $\Omega-\Omega'$ is $0$ with respect to triangulations equivalent to $|\Delta_M|.$
\end{definition}
\begin{exmp}
If $\Omega$ is smooth, then
\begin{equation}
    \int_M fd\Omega=\int_M f d\,VE_0(\Omega)\;.
\end{equation}
\end{exmp}
\section{Groupoid Construction of Functional Integrals}
Using simplicial complexes to approximate spaces in physics and mathematics is a old idea that is used in, for example, lattice gauge theory and Regge calculus, see \cite{wilson}, \cite{regge}. We use simplicial complexes in the titular groupoid approach to the functional integral, which is motivated by \cref{funco}. We won't do further analysis here, however general convergence results for QFTs on a lattice are discussed in \cite{fuji}, \cite{glimm}.
\\\\We have included a \textit{rough} correspondence of Lie algebroid$\Longleftrightarrow$ Lie Groupoid data in \ref{corr}.
\subsection{Data and Approximations}
The data we are presented with in functional integrals is often of the following form (or some minor modification of):
\begin{enumerate}
 \item An (oriented) compact $n$-dimensional manifold $M$ (with boundary),
    \item A Lie algebroid $\mathfrak{g}\to X\,,$
    \item $\omega_M\in \Gamma(\Lambda^{\bullet}TM^*)\,,$
    \item  $\omega_X\in \Gamma(\Lambda^{\bullet}\,\mathfrak{g}^*)\,,$
    \item An action functional 
    \begin{equation}
        S:\text{Hom}(TM,\mathfrak{g})\to\mathbb{R},\;S[\phi]=\int_M \mathcal{L}\;,(\omega_M,\phi^*\omega_X)
        \end{equation}
        where $ \mathcal{L}(\omega_M,\phi^*\omega_X)\in \Gamma(\Lambda^{n}\,TM^*)\;,$\footnote{We view this as the composition of maps, where $\mathcal{L}$ is real-valued.}
    \item $x_1,\ldots,x_k\in M\;,$
    \item A continuous function $F:X^k\to\mathbb{C}\,,$
    \item $\lambda\in\mathbb{C}\,.$
\end{enumerate}
In which case the functional integral is given by:
\begin{align}\label{func}
    \int_{\{\phi:TM\to\mathfrak{g}\}}\mathcal{D}\phi\,e^{\lambda S[\phi]}\,F(\phi(x_1),\ldots,\phi(x_k))\;.
\end{align}
To approximate \ref{func}, we use $VE_{\bullet}$ to integrate $\omega_M,\omega_X$ to cochain data $\Omega_M,\Omega_X$ on $ \text{Pair}\,Y_{ loc}\,,\, G_{ loc}$\footnote{It always makes sense to integrate with respect to $VE_0.$ To integrate to higher order we need a metric along the source fibers (or some splitting).} Heuristically, the degree $\bullet$ we use for $VE_{\bullet}$ is determined by the fact that the ``measure" should be concentrated on morphisms with as high a H\"{o}lder regularity as possible for which the action diverges. Note that, the local Lie groupoid always exists, see Crainic–Fernandes \cite{rui}, \cite{ruif}. Also see Cabrera–M\v{a}rcuţ–Salazar \cite{cabr} for an alternative constructon.
\\\\We get an approximation to \ref{func} by choosing a triangulation $\Delta_M$ of $M,$ forming the associated simplicial set and computing
\begin{equation}\label{est}
\int_{\{\phi:\Delta_M\to G_{\textup{loc}}\}}\mathcal{D}_{\Delta}\phi\,e^{\lambda\sum_{\Delta\in \Delta_M} \mathcal{L}(\Omega_M,\phi^*\Omega_X)(\Delta)}\,F(\phi(x_1),\ldots,\phi(x_k))\;.
\end{equation}
We then take the limit over all triangulations equivalent to $\Delta_M$ and define this to be the functional integral \ref{func}, assuming the limit exists and is independent of which equivalence class of triangulations we use to take the limit.\footnote{Of course, we are mostly ignoring any use of renormalization in this paper.}
\subsubsection{Sigma Algebra}
Heuristically, we think of $\mathcal{D}\phi\,e^{\lambda S[\phi]}$ as being a measure\footnote{A complex measure if $\lambda$ isn't real. Of course, in Brownian motion this is more than a heuristic.} on the space of \textit{continuous} Lie algebroid morphisms $\phi:TM\to\mathfrak{g},$ by which we mean the space of \textit{continuous} morphisms of local groupoids 
\begin{equation}
    \text{Pair}\,Y_{ loc}\to G_{ loc}
    \end{equation}
with respect to some $\sigma$-algebra. In Brownian motion and Feynman's path integral, this is the Borel $\sigma$-algebra of the compact open topology.
\\\\On the other hand if the theory is topological, by which we meanv $\omega_M=1$ and $\omega_X$ only consists of cocycles (see \cref{tqfts}), then the $\sigma$-algebra is the sub-$\sigma$-algebra generated by homotopy classes of maps relative to the boundary of $M$ (ie. the finest sub-$\sigma$-algebra of the Borel $\sigma$-algebra of the compact open topology such that homotopy classes of maps relative to the boundary are atoms). At the level of groupoids, this is the sub-$\sigma$-algebra of maps generated by natural isomorphism classes of maps relative to the boundary.
\subsubsection{The Measure}
The term $\mathcal{D}\phi$ (which we write as $D_{\Delta}\phi$ in the approximation) is the only term which is typically specified at the level of $\text{Hom}(\Delta_M,G_{\textup{loc}}),$ not at the level of $\text{Hom}(TM,\mathfrak{g}).$ Some natural data associated with $\mathfrak{g}, G_{\textup{loc}}$ that can help determine such a measure include measures on: the orbits, isotropy groups, leaf space, source fibers, space of arrows; Dirac measures. In the case of Brownian motion we have a Haar measure (ie. a translation invariant measure along the source fibers) determined by the Riemannian metric, and that is all we need. 
\\\\In the case of Poisson manifolds, there is a natural measure on each orbit and on the space of arrows. In the case that the symplectic groupoid has compact isotropy (eg. if the Poisson manifold is symplectic), there is a canonical Haar measure due to the fact that for a fixed source fiber the target map is a principal bundle over the orbit (for the isotropy group at the corresponding source). If instead the isotropy groups are just unimodular, eg. if the Poisson structure is regular, then there is a canonical Haar measure up to a Casimir function.
\begin{remark}
We could consider more general functional integrals than those of the form \ref{func}, however many cases of interest are of this form. Some which are not are those coming from fractional Brownian motion, see eg. \cite{meerk}. However, they are not so different, this framework may still be applicable.
\end{remark}
\subsubsection{*The Functional Integral as an Improper Integral}\label{imrop}
If we replace $G_{\textup{loc}}$ by its closure then we can triangulate $\text{hom}(\Delta_M,G_{\textup{loc}})$ and apply the cochain construction to it, using
\begin{equation}
    \mathcal{D}_{\Delta}\phi\,e^{\lambda S[\phi]}
    \end{equation}
as our form. The finite dimensional integrals are then generalized ``Riemann" integrals. However, $\text{hom}(\text{Pair}\,M_{\textup{loc}},G_{\textup{loc}})$ isn't compact and so we can't even triangulate it. This is what happens in the case of an improper Riemann integral: we can triangulate $[a,b]$ but not $\mathbb{R},$ so instead we define the integral on a sequence of compact intervals and take a limit; sometimes this construction even results in a measure.\footnote{For this discussion in the context of Brownian motion, see \cref{impropb}.}
\\\\We can think of this as follows: let $(A,\mathcal{A})$ be a measurable space and consider the category of sub-$\sigma$-algebras of trace $\sigma$-algebras with (complex) measures, ordered by inclusion, ie.
\begin{enumerate}
    \item an object is a (complex) measure space $(B,\mathcal{B},\mu_\mathcal{B})$ where $B\subset A,\,\mathcal{B}\subset \mathcal{A}\,,$
    \item there is a unique morphism $(C,\mathcal{C})\to (B,\mathcal{B})$ if $\mathcal{C}\subset\mathcal{B}$ and if $\mu_{\mathcal{B}}\vert_{\mathcal{C}}=\mu_{\mathcal{C}}\,.$
\end{enumerate}
Now, let $\mathcal{I}$ be a directed set and consider a direct system $\{(A_i,\mathcal{A}_i,\mu_{\mathcal{A}_i})\}_{i\in I}\,$ such that the direct limit of $\{(A_i,\mathcal{A}_i)\}_{i\in I}\,$ is $(A,\mathcal{A}).$ 
\\\\This makes $(A,\mathcal{A})$ into a kind of improper measure space, since the direct limit of $\{(A_i,\mathcal{A}_i,\mu_{\mathcal{A}_i})\}_{i\in I}\,$ may not exist. We can try to compute \textit{improper} integrals over $(A,\mathcal{A})$ by taking the limit of $\mu(A_i)$ in the sense of nets. 
\begin{exmp}
Let $A=\mathbb{R}$ with the Borel $\sigma$-algebra and pick a (complex) $1$-form $fdx.$ Let the directed set be compact intervals ordered by inclusion, and let $A_{[a,b]}=[a,b]$ with the Borel $\sigma$-algebra. The direct limit of this system is $\mathbb{R}$ and we can pull back $fdx$ to a (complex) measure on $[a,b],$ in which case the limit of $\int_a^b f\,dx$ is the improper Riemann integral. Note that, the direct limit exists as a complex measure space if and only if $|f|\in L^1(\mathbb{R}).$
\end{exmp}
\begin{exmp}
Let $A=\textup{hom}(\textup{Pair}\,\textup{M}_{\textup{loc}},G_{\textup{loc}}),$ with $\mathcal{A}$ being the Borel $\sigma$-algebra of the compact open topology and fix an equivalence class of triangulations on $M;$ this will be our directed set.
\\\\Let $A_{\Delta_M}=\textup{hom}(\textup{Pair}\,\textup{M}_{\textup{loc}},G_{\textup{loc}})$ and let $\mathcal{A}_{\Delta_M}$ be the finest sub-$\sigma$-algebra that doesn't separate morphsisms which agree on $\Delta_M,$ ie. a set $B\in\mathcal{A}_{\Delta_M}$ if and only if $B\in\mathcal{A}$ and if $f,g$ agree on $\Delta_M$ then $f\in B\implies g\in B.$
\\\\
The direct limit of this system is $\mathcal{A}.$ Suppose we have our corresponding (complex) measures  $\mathcal{D}_{\Delta}\phi\,e^{\lambda S[\phi]};$ these won't necessarily form a direct system. However, for $B\in \mathcal{A}_{\Delta_M'}$ we can define (if it converges)
\begin{equation}
    \mu_{\Delta'_M}(B)=\lim\limits_{\Delta_M}\int_B\mathcal{D}_{\Delta}\phi\,e^{\lambda S[\phi]}\;,
\end{equation}
and this gives us a direct system of measure spaces (if all of these integrals converge). We can use this direct system to compute the \textit{improper} functional integrals.
\\\\In \cref{direct} we showed that Wiener space is a direct limit of such measure spaces. For the Feynman path integral, it is well-known that the direct limit doesn't exist as a complex measure space.
\end{exmp}
\subsection{Examples}
Here we will show some examples of the previous construction.
\begin{exmp}
As described in the introduction, Brownian motion and Feynman's path integral can be described in this way. If we want to include a magnetic potential $A(x)\,dx,$ then as discussed in \cref{physics} the action we really want to consider is
\begin{equation}
    S[\mathbf{x}]=\int_0^1 \frac{1}{2}\dot{\mathbf{x}}^2\,dt-V(\mathbf{x})\,dt+A(\mathbf{x})\dot{\mathbf{x}}\,dt+\frac{1}{2}A'(\mathbf{x})\dot{\mathbf{x}}^2 \,dt^2\,.
\end{equation}
Classsically, ie. for smooth paths, this action is indistinguishable from
\begin{equation}
    S[\mathbf{x}]=\int_0^1 \frac{1}{2}\dot{\mathbf{x}}^2\,dt-V(\mathbf{x})\,dt+A(\mathbf{x})\dot{\mathbf{x}}\,dt\;.
\end{equation}
However with respect to quantum (or Brownian) paths, they are distinct, since these paths have H\"{o}lder exponent in $(1/3,1/2)$ on a set of full measure. Therefore, we must use $VE_1$ to compute the cochains, rather than $VE_0.$ That is, we need a $1$-cochain $\Omega$ on $\textup{Pair}\,\mathbb{R}$ such
\begin{equation}\label{ho}
    VE_1(\Omega)=A(x)\,dx+\frac{1}{2}A'(x)\,dx^2\;.
\end{equation}
As discussed in \cref{physics}, any completely antisymmetric $1$-cochain $\Omega$ such that $VE_0(\Omega)=A(x)\,dx$ will satisfy \ref{ho}. For example, we can take $\Omega$ to be either of
\begin{align}\label{pote}
& A((x+y)/2)(y-x)\,,
\\& \nonumber \frac{1}{2}(A(x)+A(y))(y-x)\;.
\end{align}
We still need to specify the cochain data for $dx^2,\, V(x),\, dt.$ The most obvious choices are, respectively,
\begin{align}
&\nonumber(x,y)\mapsto (y-x)^2\;,
\\&\nonumber (x,y)\mapsto V(x)\;,
\\ & (t_0,t_1)\mapsto t_1-t_0\;.
\end{align}
Using \ref{pote} for $A(x)\,dx$ and letting $0=t_0<t_1<\cdots< t_N=1$ be an evenly spaced triangulation $\Delta_{[0,1]},$ the approximation to the path integral is
\begin{align}
& \nonumber\psi(x,1)=\lim\limits_{N\to\infty}\sqrt{\frac{mN}{2\pi i\hbar}}\int_{\mathbb{R}^N}\prod_{n=0}^{N-1}d\mathbf{x}_n\,e^{\frac{1}{N}\frac{i}{\hbar}\sum_{n=0}^{N-1} S[x_n,x_{n+1}]}\psi_i(\mathbf{x}_{N-1})\;\,,
\\&\nonumber
\\& S[x_{n},x_{n+1}]=\frac{m}{2}N^2(\mathbf{x}_{n+1}-\mathbf{x}_{n})^2-V(\mathbf{x}_{n})+N\,A((x_{n+1}+x_n)/2)(x_{n+1}-x_n)\;.
\end{align}
This agrees with the usual formulation, see \cite{gav} (of course, we are free to pick from more general cochain data).
Here, $dx$ is the Haar measure on $\textup{Pair}\,\mathbb{R}$ determined by the Euclidean metric. Note that, we have used the fact that
\begin{equation}
\{\mathbf{x}\in\textup{Hom}(\Delta_{[0,1]},\textup{Pair}\,\mathbb{R})\,:\mathbf{x}(0)=x_i\}=\{(x_0,x_1,\ldots,x_N)\in\textup{Pair}^{(N)}\,\mathbb{R}\,: x_0=x_i\}\;.
\end{equation}
\end{exmp}
\begin{exmp}
Here we exhibit a kind of generalized Riemann-Lebesgue duality. Let $M$ be a manifold and consider a top form $\omega.$ Let $*$ be a point. There are two natural functional integrals associated with this data. We can consider:
\begin{enumerate}
    \item $\textup{hom}(TM,*)=*\,,$
    \item $\textup{hom}(*,TM)=M\,.$
\end{enumerate}
Applying our framework to the former gives us our construction of $\int_M \omega.$ Applying our framework to the latter gives us the (signed) measure corresponding to $\omega.$ 
\end{exmp}
\begin{exmp}\label{freed}
This is a topological example of Chern-Simons with finite gauge group, as explained in Dijkgraaf–Witten \cite{witten} and Freed–Quinn \cite{freed}. Formally, it fits into our framework exactly as how we've discussed it, however we need to allow our cochains to be $S^1$-valued, which is exactly what one gets after exponentiating the action.
\\\\At the level of Lie algebroids, we need to consider $H^*(TX,S^1),$ as in \cref{part2}.
Geometrically, these are gerbes with flat connections. The action is given by pulling back and computing the monodromy:
\\\\
For dimension reasons, an $S^1$-valued $n$-cocycle $\Omega$ on the local pair groupoid of an $n$-dimensional manifold is cohomologically equivalent to $e^{i\Omega_0}\,,$ for some $\mathbb{R}$-valued $n$-cocycle $\Omega_0.$ Similarly, $VE_0(\Omega)$ is cohomologically equivalent to $VE_0(\Omega_0).$ We have an $S^1$ -valued ``integral" (Riemann sum) obtained by exponentiating the previously defined integral (Riemann sum) of $VE_0(\Omega_0)\;(\Omega_0).$ By \ref{main2}, the exponentiated Riemann sum is exact.
\\\\Now we will do the example. 
\\\\Let $G\rightrightarrows *$ be a finite group with classifying space $BG.$ Let $\omega$ be an $S^1$-valued $n$-cocycle on $TBG$\footnote{We don't really need to concern ourselves with infinite dimensional analysis, since the groupoid is what is fundamental in this framework.} The action of a map $f:M\to BG,$ where $M$ is closed and $n$-dimensional, is still denoted $S[f]$ and is given by pulling back and integrating $\omega,$ as described in the previous paragraph. We integrate $\omega$ to an $S^1*$-valued $n$-cocycle $\Omega$ on $\Pi_1(BG),$ with corresponding action of $f:\Delta_M\to \Pi_1(BG)$ also denoted $S[f]$ (note that, $\Pi_1(BG)$ has contractible source fibers, so we use it instead of a local pair groupoid).
\\\\The $\sigma$-algebra $\mathcal{B}$ on $\textup{Hom}(M,BG)$ is the finest sub-$\sigma$-algebra of the compact-open topology such that if $B\in\mathcal{B}$ and $f$ is homotopic to $g,$ then $f\in B\implies g\in B.$ This means that homotopy classes of maps generate $\mathcal{B},$ and there are only finitely many.
\\\\We take the measure $\mu([f])$ of a homotopy class $[f]$ to be the inverse of the number of automorphisms of a map $f\in [f],$ ie. the number of homotopy classes of maps $h:M\times S^1\to BG$ such that $h(m,\star)=f(m)$ for some marked point $*\in S^1$ and all $m\in M.$ The partition function, ie. the expectation value of $1,$  is then
\begin{equation}
    Z_M=\int_{\textup{Hom}(M,BG)}S\,d\mu=\sum_{[f]\in\textup{Hom}(M,BG)/\sim}S[[f]]\,\mu([f])\;.
\end{equation}
This is the result obtained in \cite{witten}, \cite{freed}, but we think about it a bit differently. In particular, our domains of integration are different and our cocycle data is specified differently. We can specify the cocycle data as in \cite{witten} if we allow our hom spaces to contain generalized morphisms, which our framework naturally exends to (at least when the hom spaces are principal bundles with flat connections).
\end{exmp}
\begin{exmp}\label{psm}
This is a slight generalization of the example appearing in Freed–Hopkins–Lurie–\\Teleman \cite{Lurie} and is a topological example.
\\\\Consider a Lie algebra $\mathfrak{g}\to *$ and a morphism $\omega:\mathfrak{g}\to\mathbb{R},$ which integrates to $\Omega:G\to \mathbb{R},$ where we assume that $G$ is simply connected and amenable (eg. compact or abelian)\footnote{For a definition of amenable, see \cite{Runde}.} whose finitely additive, invariant probability measure we denote by $dg.$ We want to consider the partition function of $S^1.$
\\\\Any morphism $f:TS^1\to\mathfrak{g}$ integrates to a unique morphism 
\begin{equation}
    \Pi_1(S^1)\to G,
    \end{equation}
and we can pull back the integrated morphism to a morphism $F\in\textup{Hom}(\Delta_{S^1},G),$ for any triangulation of $S^1.$ Then by \cref{main2}, similarly to the previous example,
\begin{equation}\label{lur}
   \sum_{\Delta\in\Delta_{S^1}}F^*\Omega= \int_{S^1}f^*\omega\;.
\end{equation}
Now, the set of homotopy classes of maps $TS^1\to \mathfrak{g}$ is naturally identified with $G,$ and using this identification we have that
\begin{equation}
    \int_{S^1}f^*\omega=\Omega(g)\;.
\end{equation}
Hence, the partition function (ie. the expectation value of $1$) is
\begin{equation}
    \int_G e^{i\Omega(g)}\,dg\;.
\end{equation}
This is equal to $1$ if $\,\Omega=0$ and is $0$ otherwise. 
\end{exmp}
\subsubsection{Poisson Sigma Model Examples}
Consider the Poisson sigma model, with domain the disk with three marked points on the boundary, denoted $0,1,\infty.$ The target space of this sigma model is a Poisson manifold $(M,\Pi)$ and there are two fields, given by a vector bundle morphism 
\begin{equation*}
    (X,\eta):TD\to T^*M
\end{equation*}
which on $\partial D$ takes values in the zero section. The action, defined in Schaller–Strobl \cite{strobl}, is given by 
\begin{equation}\label{action}
    S[X,\eta]=\int_{D} \eta\wedge dX+\frac{1}{2}\Pi\vert_X(\eta\,,\eta)\,.
\end{equation}
In Cattaneo–Felder \cite{catt} it is argued that Kontsevich's star product \cite{kontsevich} is equal to the perturbative expansion around the trivial classical solution of the following path integral:
\begin{equation}\label{star}
 (f\star g)(m)=\int_{X(\infty)=m} f(X(1))g(X(0))\,e^{\frac{i}{h}S[X,\eta]}\,\mathcal{D}X\,\mathcal{D}\eta\;,
\end{equation}
normalized so that $1\star 1=1.$ Furthermore, it was shown in Bonechi–Cattaneo–Zabzine \cite{bon} that (formally) this integral is equal to
\begin{align}\label{mystarr}
    (f\star g)(m)=\int_{\{X\in\textup{Hom}(TD, T^*M):\pi(X(\infty))=m\}}f(\pi(X(1)))\,g(\pi(X(0)))\,e^{\frac{i}{\hbar}S[X]}\,\mathcal{D}X\,,
\end{align}
where $\pi:T^*M\to M$ is the projection and the integral is over Lie algebroid morphisms. The action is given by
\begin{equation}\label{actp}
S[X]=\int_D X^*\Pi\,.
\end{equation}
The theory with action \ref{actp} fits into our framework and is topological. More generally, we can consider this action for any surface.
\\\\It was shown by the author in \cite{Lackman4} that $\ref{mystarr}$ is the twisted convolution algebra of the canonical geometric quantization of the Lie 2-groupoid $\Pi_2(\mathfrak{g}).$ This Lie 2-groupoid was studied by Zhu in \cite{zhuc}.
\begin{exmp}
In the context of example \ref{psm}, let's consider conventional quantum mechanics on  $T^*\mathbb{R}\cong\mathbb{R}^2,$ with symplectic structure $\omega=dp\wedge dq.$ We will use triangulations of $D$ where the marked points $0,1,\infty$ are vertices.
\\\\Let's consider the simplest triangulation, where we identify the disk with the standard 2-simplex $|\Delta^2|,$ and denote the trivial triangulation by $\Delta^2.$ The groupoid is $\textup{Pair}\,\mathbb{R}^2,$ and the cocycle is 
\begin{equation}
    \Omega(p,q,p_1,q_1,p_2,q_2)=\begin{vmatrix} p_{1}-p & p_2-p \\ q_{1}-q & q_{2}-q \end{vmatrix}\;,
\end{equation}
ie. $VE_0(\Omega)=\omega$ (in fact, with the standard metric $VE_{\infty}(\Omega)=\omega).$ We want to approximate $(f\star g)(p,q),$ so let's fix a $(p,q).$ The space we want to integrate over is 
\begin{equation}
    \{\mathbf{x}\in\textup{Hom}(\Delta^2,\textup{Pair}\,\mathbb{R}^2):\mathbf{x}(\infty)=(p,q)\}\;,
    \end{equation}
    which is isomophic to
\begin{equation}
    \{(p,q,p_1,q_1,p_2,q_2):(p_1,q_1,p_2,q_2)\in\mathbb{R}^4\};.
    \end{equation}
We have a canonical measure on this space induced by $\omega.$  After normalizing so that $1\star 1=1,$ the resulting integral is
\begin{align}\label{moyal}
(f\star g)(p,q)=\frac{1}{(2\pi\hbar)^2}\int_{\mathbb{R}^4}f(p_2,q_2)g(p_1,q_1)e^{\frac{i}{\hbar}\Omega(p,q,p_1,q_1,p_2,q_2)}\,dp_1\,dq_1\,dp_2\,dq_2\;. 
\end{align}
This approximation is exact, ie. \ref{moyal} is the non-perturbative form of the Moyal product, see \cite{baker}, also \cite{Zachos}. It is associative and satisfies 
\begin{equation}
    f\star g-g\star f=i\hbar\{f,g\}+\mathcal{O}(\hbar^2)\;.
\end{equation}
It provides a strict deformation quantization of $T^*\mathbb{R}^2,$ in Rieffel's sense, see \cref{strict}. The operator assignment \begin{equation}
    f\mapsto \hat{f}\,,\;\hat{f}g:=f\star g\;,
    \end{equation}
for $g\in L^2(T^*\mathbb{R},\omega),$ is essentially the non-perturbative form of the Weyl quantization of $f.$
\\\\One can get this product from geometric quantization as well, but to do so  requires a connection, polarization and the Fourier transform, see eg. \cite{eli}.
\end{exmp}
\begin{remark}
One can show that \ref{est} converges to \ref{moyal} when one takes the limit over triangulations with an even number of vertices between the marked points, however we believe it is more natural to shrink the local groupoid with the triangulations. It is not known what happens in this case.
\end{remark}
\begin{exmp}\label{tang}
 Similar to the previous example, one can show that the approximation is exact,\footnote{The previous remark still applies here.} and thus gives strict deformation quantizations (which we discuss in \cref{deform}), in the following cases:
 \begin{enumerate}
     \item  $0$-Poisson structure on a manifold,
     \item the dual of the Lie algebra of the Heisenberg group,
     \item  symplectic manifolds whose universal cover is $T^*\mathbb{R}^n,$ 
     \item Poisson structures on manifolds obtained by an action of $\mathbb{R}^n\rightrightarrows *$ with the constant Poisson structure, see Rieffel \cite{reiffelact},
     \item Poisson structures on manifolds obtained by an action of the two dimensional affine group with its invariant Poisson structure, see Bieliavsky–Bonneau–Maeda\cite{pierre}.
 \end{enumerate}     
To explain examples $4$ and $5,$ let $\Pi$ be an invariant Poisson structure on a Lie group $G\rightrightarrows *.$ Then given an action of $G$ on a manifold $M,$ we get a morphism $\mathfrak{g}\to TM$ and we can use this morphism to push forward the Poisson structure to $M.$
\\\\Example $2$ was computed by the author in \cite{Lackman4}. The result is
\begin{align}\label{hei}
    &(f\star g)(x,y,z)=\nonumber
    \\&\frac{1}{(2\pi z\hbar)^2}\int_{\mathbb{R}^2\times\mathbb{R}^2}f(x'',y'',z)g(x',y',z)e^{\frac{i}{z\hbar}((x''-x)(y'-y)-(x'-x)(y''-y))}\,dx''\,dy''\,dx'\,dy'\;,
\end{align}
with $(f\star g)(x,y,0)=f(x,y,0)g(x,y,0).$ This product is equal to the product Reiffel constructed on the dual of nilpotent Lie algebras in \cite{rieffel2}, in this special case. For fixed $z$ this is the Moyal product. 
\\\\For example $4,$ consider an action of $\mathbb{R}^n$ on a manifold $M,$ denoted $m\mapsto x\cdot m.$ A constant Poisson structure on $\mathbb{R}^n$ is determined by a skew-symmetric matrix $J$ on $\mathbb{R}^n.$ One can show that applying our framework gives Rieffel's universal deformation formula:
\begin{equation}
(f\star g)(m)=\frac{1}{(2\pi\hbar)^n}\int_{\mathbb{R}^n\times \mathbb{R}^n}f(x\cdot m)g(Jy\cdot m)e^{\frac{i}{\hbar}x\cdot y}\,d^nx\,d^ny\;.
\end{equation}
A similar formula holds for example $5.$ It would be interesting to determine if other known examples of tangential, strict deformation quantizations\footnote{See \cref{deform}.} of Poisson manifolds have exact approximations, eg. \cite{bone}.
\\\\In general, there are canonical $2$-cocycles on the symplectic groupoid for the duals of exponential Lie algebras (ie. where the exponential map is a diffeomorphism) and the duals of semisimple Lie algebras; see the examples section in \cite{Lackman4}.
\end{exmp}
\subsection{*Non-Perturbative Deformation Quantizations}\label{deform}
Observe that, the star product in \ref{mystarr} is tangential, meaning that it pulls back to a star product on symplectic leaves, see eg. \cite{wein}, \cite{amar}, \cite{cahen}, \cite{gamme}. This is because a Lie algebroid morphism remains in a single orbit. Therefore, if $f_1,g_1, f_2,g_2$ are such that $f_1=f_2,\,g_1=g_2$ on a symplectic leaf $L,$ then $f_1\star g_1=f_2\star g_2$ on $L$ as well. 
\\\\There are Poisson structures which are known to not have tangential deformation quantizations, so it would seem that \ref{mystarr} can't be equal to \ref{star} in full generality. However, it is true that regular Poisson structures always have tangential deformation quantizations, so as far as the author is aware it is consistent with the literature that these are equal for regular Poisson structures, see Masmoudi.\cite{Masmoudi}. The dual of the Lie algebra of the Heisenberg group is an example of a Poisson structure which isn't regular that has a tangential quantization.
\\\\We conjecture that there is a cocycle for which the approximation is exact for any tangential and strict deformation quantization (say, 
 a cocycle obtained using $VE_{\infty}),$ though we don't expect that there are many of these beyond the ones we listed in example \ref{tang}. All of the examples we listed have the property that the universal cover of each symplectic leaf is $T^*\mathbb{R}^n.$ 
 \\\\We recall Rieffel's definition of strict deformation quantization \cite{rieffel}:
 \begin{definition}\label{strict}
Let $(M,\Pi)$ be a Poisson manifold. A strict deformation quantization of $(M,\Pi)$ is, for each $\hbar\in\mathbb{R}\,,$ a $C^*$-algebra containing $C^{\infty}_c(M)$ as a dense $^{*_{\hbar}}$-subalgebra, denoted 
\vspace{0.05cm}\\
\begin{equation*}
    (C^{\infty}_c(M),\,\star_{\hbar}, \,^{*_{\hbar}},\,||\cdot ||_{\hbar})\,,
    \end{equation*}
such that
\begin{enumerate}
    \item For each $f\in C^{\infty}_c(M)\,,$ the map $\hbar\to ||f||_{\hbar}$ is continuous.
    \item For $\hbar=0$ the $C^*$-algebra is the natural one, ie. $(C^{\infty}_c(M),\cdot, ^*,||\cdot ||_{L^{\infty}(M)})\,,$
    where the product is multiplication and $^*$ denotes complex conjugation.
    \item $||(f\star_{\hbar}g-g\star_{\hbar}f)/\hbar-i\{f,g\}||_{\hbar}\xrightarrow[]{\hbar\to 0}0\,.$
\end{enumerate}
\end{definition}
We conjecture the following: 
\begin{conjecture}
Let $(M,\omega)$ be a symplectic manifold. Then there is a strict deformation quantization of $C^{\infty}_c(M),$ whose underlying $C^*$-algebra is a subalgebra of bounded linear operators on $L^2(M,\omega),$ if and only if its universal cover is $T^*\mathbb{R}^n.$ 
\end{conjecture}
Indeed, this is in a similar vein to an observation made by Weinstein in \cite{weinstein}, which is that in order for his geometric quantization scheme to work for a Poisson manifold $(M,\Pi)$ it seems necessary that the symplectic groupoid is diffeomorphic to $T^*M.$ Specializing this to the case of symplectic manifolds, his observation implies that the universal cover of $M$ should be contractible.
\\\\This conjecture is also consistent with a no-go theorem proved by Rieffel, which says that there is no $SO(3)$-equivariant, strict deformation quantization of $(S^2,\omega)$ with its standard symplectic structure, see point 14 in \cite{rieffel3}. This means that if there is one, it has to be anomalous. This sheds some doubt on the existence of an associative and non-perturbative form of Kontsevich's star product on $(S^2,\omega).$ 
\\\\Given this information, one might wonder what extra structure a non-perturbative definition of \ref{star} would have. We can relax Rieffel's definition of strict deformation quantization in such a way that we still get a deformation quantization and a $C^*$-algebra. This can be done by not requiring $C^{\infty}_c(M)$ to be a subalgebra in \cref{strict}, but strengthening the third condition so that it looks like a subalgebra perturbatively:
\begin{definition}
Let $(M,\Pi)$ be a Poisson manifold. For functions $f:M\to\mathbb{C}$ such that 
\begin{equation}
    f\vert_{\mathcal{L}}\in L^2(\mathcal{L},\Pi\vert_{\mathcal{L}}^{-1})
    \end{equation}
for all symplectic leaves $\mathcal{L},$ let
\begin{equation}
    ||f||:=\text{sup}_{\mathcal{L}}||f||_{\mathcal{L}}\;,
\end{equation}
where $||f||_{\mathcal{L}}$ is the $L^2$ norm of $f\vert_{\mathcal{L}}$ over the symplectic leaf. Denote the space of functions with bounded norm by $L^2(M,\Pi).$
\end{definition}
\begin{definition}
A non-perturbative deformation quantization of a Poisson manifold $(M,\Pi)$ is given by a pair $(\star_{\hbar}, \bullet_{\hbar}),$ where $\star_{\hbar}$ is a star product on $C_c^{\infty}(M)$ and where
\begin{equation}
    \bullet_{\hbar}:L^2(M,\Pi)\times L^2(M,\Pi)\to L^2(M,\Pi)
\end{equation}
is a continuous bilinear map. We get an operator assignment given by $f\mapsto\hat{f},\;\hat{f}g:=f\bullet_{\hbar}g.$ We require that:
\begin{enumerate}
\item $\overline{f\bullet_{\hbar} g}=\bar{g}\bullet_{\hbar} \bar{f},$ where $\;\bar{}\;$ is the complex conjugate,
    \item For $f,g\in C^{\infty}_c(M),$
    \begin{equation}
        \frac{1}{\hbar^n}||\widehat{f\bullet_{\hbar}g} -\widehat{T_n[f\star_{\hbar}g]}\,||\xrightarrow[]{\hbar\to 0} 0
    \end{equation}
    for all $n\ge 0,$ where $T_n$ truncates the formal power series at degree $n$ (ie. we use the operator norm).
\end{enumerate}
\end{definition}
The above definition can be weakened. In particular, in condition $2$ we can instead only require that
\begin{equation}
    \frac{1}{\hbar^n}||f\bullet_{\hbar}g -T_n[f\star_{\hbar}g]\,||\xrightarrow[]{\hbar\to 0} 0\;.
\end{equation}
\subsection{Lie Algebroid:Lie Groupoid Dictionary}\label{corr}
Here we include a \textit{rough} correspondence of Lie algebroid$\Longleftrightarrow$ Lie Groupoid data.
\begin{center}
\begin{tabular}{ | m{5cm} | m{7cm}| } 
  \hline
  Lie Algebroids & Lie Groupoids\\ 
  \hline
  $\mathfrak{g}\to X$& (local) Lie groupoid $G\rightrightarrows X$   \\ 
  \hline
  morphisms $\mathbf{x}:TM\to\mathfrak{g}$ & morphisms $\{\mathbf{x}:\Delta_M\to G\}_{\Delta_M\in \mathcal{T}_M}$\\
  \hline
  metric, symplectic form, etc. & measures on spaces $\{\mathbf{x}:\Delta_M\to G\}_{\Delta_M\in \mathcal{T}_M}$\\
  \hline
  $\mathfrak{g}^{\oplus n}\to X$ & $G^{(n)}\xrightarrow[]{s}X$\\
  \hline
 (closed) n-form & (closed) completely antisymmetric $n$-cochain \\ 
  \hline
    (completely) symmetric n-form & normalized (completely) symmetric n-cochain\\
  \hline
  Riemannian metric & positive, normalized symmetric 2-cochain\\
  \hline
  Poisson manifold & symplectic groupoid\\

  \hline
\end{tabular}
\end{center}
 \begin{appendices}\label{appen}
 \section{Basic Theory of Lie Groupoids and Lie Algebroids}
In this section we will begin by describing Lie groupoids and Lie algebroids and we will give some important examples. For some lecture notes, see \cite{eck}. For more about simplicial manifolds, their cohomology, etc. see Deligne \cite{Deligne}.
\subsection{Lie Groupoids and Lie Algebroids}
\begin{definition}
A groupoid is a category $G\rightrightarrows X$ for which the objects $X$ and arrows $G$ are sets and for which every morphism is invertible. Notaionally,, we have two sets $X, G$ with structure maps of the following form:
\begin{align*}
    & s,t:G\to X\,,
   \\ & \textup{id}:X\to G\,,
    \\ & \cdot:G\sideset{_t}{_{s}}{\mathop{\times}} G\to G\,,
    \\& ^{-1}:G\to G\,.
\end{align*}
Here $s,t$ are the source and target maps, $\textup{id}$ is the identity bisection (ie. $X$ can be thought of as the set of identity arrows inside $G$), $\cdot$ is the multiplication, denoted $(g_1,g_2)\mapsto g_1\cdot g_2,$ and $^{-1}$ is the inversion map. We will frequently identify a point $x\in X$ with its image in $G$ under $\textup{id}$ and write $x\in G.$  
\\\\A Lie groupoid is a groupoid $G\rightrightarrows X$ such that $G, X$ are smooth manifolds, such that all structure maps are smooth and such that the source and target maps submersions.
\end{definition}
For brevity, we will sometimes denote a (Lie) groupoid $G\rightrightarrows X$ exclusively by its space of arrows $G.$
\begin{definition}
A morphism of groupoids $G\to H$ is a functor between them, ie. a function which is compatible with the multiplcations. A morphism of Lie groupoids is a functor which is smooth.
\end{definition}
\begin{exmp}
Any Lie group $G$ is a Lie groupoid $G\rightrightarrows \{e\}$ over the manifold containing only the identity $e\in G.$
\end{exmp}
The following example is the one most relevant to Brownian motion:
\begin{exmp}\label{pair}
 Let $X$ be a manifold. The pair groupoid, denoted 
 \begin{equation*}
     \textup{Pair}(X)\rightrightarrows X\,,
      \end{equation*}
    is the Lie groupoid whose objects are the points in $X$ and whose arrows are the points in $X\times X.$ An arrow $(x,y)$ has source and target $x,y,$ respectively. Composition is given by $(x,y)\cdot(y,z)=(x,z),$ the identity bisection is $\textup{id}(x)=(x,x)$ and the inversion is $(x,y)^{-1}=(y,x).$
\end{exmp}
\begin{exmp}\label{fundg}
Let $M$ be a manifold. The fundamental groupoid, denoted $\Pi_1(M)\rightrightarrows M,$ is a Lie groupoid whose arrows between two objects $m_1,m_2$ are homotopy classes of smooth maps starting at $m_1$ and ending at $m_2,$ with composition given by concatenation. The space of arrows is naturally identified with $\tilde{M}\times\tilde{M}/\pi_1(M),$ where $\tilde{M}$ is the universal cover of $M$ and $\pi_1(M)$ acts by the diagonal action. This groupoid integrates $TM$ and has simply connected source fibers. Its local groupoid is isomorphic to the local pair groupoid of $M.$
\end{exmp}
\begin{exmp}
Let $G$ be a Lie group acting on a manifold $M.$ We obtain a Lie groupoid
\begin{equation}
    G\ltimes M\rightrightarrows M,
\end{equation}
called the action groupoid. The source and target of $(g,m)$ are $m,\,g\cdot m,$ respectively. The composition is given by $(g_1,m)\cdot(g_2,g_1\cdot m)=(g_2g_1, m).$
\end{exmp}
The infinitesimal counterpart of a Lie groupoid is a Lie algebroid.
\begin{definition}
A Lie algebroid is a triple ($\mathfrak{g},[\cdot,\cdot],\alpha)$ consisting of 
\begin{enumerate}
    \item A vector bundle $\pi:\mathfrak{g}\to X\,,$
    \item A vector bundle map (called the anchor map) $\alpha:\mathfrak{g}\to TX\,,$
    \item A Lie bracket $[\cdot,\cdot]$ on the space of sections $\Gamma(\mathfrak{g})$ of $\pi:\mathfrak{g}\to X,$
\end{enumerate}
such that for all smooth functions $f$ and all $\xi_1,\xi_2\in \Gamma(\mathfrak{g})$ the following Leibniz rule holds: $[\xi_1,f\xi_2]=(\alpha(\xi_1)f)\xi_2+f[\xi_1,\xi_2]\,.$
\end{definition}
\begin{exmp}
Any Lie algebra $\mathfrak{g}$ is a Lie algebroid $\mathfrak{g}\to \{0\}$ over the manifold containing only $0\in \mathfrak{g}.$
\end{exmp}
The following example is the one relevant to Brownian motion.
\begin{exmp}
Let $X$ be a manifold. The tangent bundle $TX\to X$ is a Lie algebroid, where the anchor map $\alpha$ is the identity. Sections in $\Gamma(TX)$ are just vector fields and the Lie bracket is Lie bracket of vector fields.
\\\\More generally, the vectors tangent to the leaves of a foliation of a manifold $X$ form a Lie algebroid, with the Lie bracket being the Lie bracket of vector fields.
\end{exmp}
\begin{exmp}
Let $(M,\Pi)$ be a Poisson manifold. Then there is a natural Lie algebroid structure on $T^*M\to M,$ with anchor map given by contraction with $\Pi$ ($\,\Pi$ is a bivector). The Lie bracket is determined by $[df,dg]:=d\,\Pi(df,dg).$ If $\Pi$ is symplectic then this Lie algebroid is isomorphic to $TM.$
\end{exmp}
Morphisms of Lie algebroids are slightly cumbersome to define, however we won't really be needing them in this paper. By Lie's second theorem, they are vector bundle maps which are obtained by Lie groupoid morphisms via differentiation. We give some examples:
\begin{exmp}
Let $M, X$ be manifolds. A morphism $TM\to TX$ is the same as the differential of a smooth map $M\to X.$
\end{exmp}
Of course, in the special case of Lie algebras we get the expected result:
\begin{exmp}
If $\mathfrak{g}, \mathfrak{h}$ are Lie algebras, then $\mathfrak{g}\to\mathfrak{h}$ is a morphism of Lie algebroids if and only if it is a morphism of Lie algebras.
\end{exmp}
One of the most important classes of Lie algebroid morphisms are the following:
\begin{exmp}
Let $M$ be a one dimensional manifold and let $\mathfrak{g}\to X$ be a Lie algebroid. Given a vector bundle map $f:TM\to\mathfrak{g},$ let $f\vert_{M}:M\to X$ be the induced map on the base. Then $f:TM\to\mathfrak{g}$ is a Lie algebroid morphism if and only if $df\vert_M:TM\to TX$ is equal to $\alpha\circ f.$ If $M=[0,1],$ such a path is called an algebroid path.
\end{exmp}
\subsection{The van Est Map}
We now state the definition of the van Est map given by Weinstein–Xu in \cite{weinstein1}. The description of the nerve that they use is 
\begin{equation}
    G^{(n)}=\underbrace{G\sideset{_t}{_{s}}{\mathop{\times}} G \sideset{_t}{_{s}}{\mathop{\times}} \cdots\sideset{_t}{_{s}}{\mathop{\times}} G}_{n \text{ times}}\;.
\end{equation}
Let $G\rightrightarrows X$ be a Lie groupoid. Given $X\in\Gamma(\mathfrak{g})\,,$ we can left translate it to a vector field $L_X$ on $G^{(1)}\,.$  Now suppose we have an $n$-cochain $\Omega,$ $n\ge 1\,.$ We get an $(n-1)$-cochain $L_X\Omega$ by defining
\begin{equation}
    L_X\Omega(g_1,\ldots, g_{n-1})=L_X\Omega(g_1,\ldots, g_{n-1},\cdot)\vert_{t(g_{n-1})}\,,
\end{equation}
ie. we differentiate it in the last component and evaluate it at the identity $t(g_{n-1})\,.$ 
\\\\The following definition uses sections of $\mathfrak{g},$ so one must check that it defines an $n$-form:
\begin{definition}\label{vanest}
Let $\Omega$ be a normalized $n$-cochain and let $X_1,\ldots X_n\in\Gamma(\mathfrak{g})\,.$ We define
\begin{align}
    VE(\Omega)(X_1,\ldots, X_n)=\sum_{\sigma\in S_n} \textup{sgn}(\sigma)L_{X_{\sigma(1)}}\cdots L_{X_{\sigma(n)}}\Omega\;.
\end{align} 
\end{definition}
\subsection{Lie's Theorems for Groupoids}
Generalizing the case of Lie groups, a Lie algebroid $\mathfrak{g}\to X$ is associated to every Lie groupoid $G\rightrightarrows X.$ The construction is formally the same as in the case of Lie groups: the underlying vector bundle $\mathfrak{g}\to X$ is given by the normal bundle of $\textup{id}(X)\subset G\,.$ This vector bundle can be identified with $\textup{id}^*\textup{ker}(ds)\,,$ ie. vectors tangent to the source fibers at the identity bisection. The anchor map is given by $dt$ and the Lie bracket is obtained in the same way it is obtained on Lie groups: by left translating sections in $\Gamma(\mathfrak{g})$ to vector fields in $TG$ tangent to the source fibers (via the derivative of the multiplication map), and evaluating the Lie bracket of vector fields at the identity bisection. 
\\\\Lie's second theorem holds for Lie groupoids, ie. given a morphism $f:G\to H$ we can differentiate it to get a corresponding morphism $df:\mathfrak{g}\to\mathfrak{h},$ and this correspondence is a bijection. However, Lie's third theorem fails: not every Lie algebroid integrates to a Lie groupoid, though they often do in practice, and they always do locally, which is enough for the constructions we will be making.
\\\\Despite the faiulre of Lie's third theorem, a Lie algebroid $\mathfrak{g}\to X$ can always be integrated to a topological groupoid which is smooth in a neighborhood of the identity bisection.\footnote{Though it may fail to be globally Hausdorff.} The idea is similar to the construction of the universal cover of a manifold: given a Lie algebroid $\mathfrak{g}\to X,$ its canonical topological integration $\Pi_1(\mathfrak{g})$ has arrows which are algebroid paths up to homotopy; these can be composed in a natural way. We will make this precise with the next proposition. 
\\\\The details of this section, including the following proposition, can be found in Crainic–Fernandes \cite{rui}:
\begin{proposition}
Let $G\rightrightarrows X$ be a Lie groupoid with simply connected source fibers, and let $\mathfrak{g}\to X$ be its Lie algebroid. We have a natural bijection 
\begin{equation}\label{path}
    G\to \textup{hom}(T[0,1],\mathfrak{g})/\sim
\end{equation}
defined as follows (here $\sim$ identifies two algebroid paths if they are homotopic relative to the endpoints):
for each $g\in G$ choose a path 
\begin{equation}
    \gamma:[0,1]\to s^{-1}(s(g))
    \end{equation}
such that $\gamma(0)=s(g),\,\gamma(1)=g.$ The map in \ref{path} is given by differentiating $\gamma$ and left translating the vectors in the image to the identity bisection, ie. for a vector $V\in T[0,1]$ at the point $t\in [0,1],$ 
\begin{equation}
    V\mapsto (\gamma(t))^{-1}\cdot d\gamma(V)\;.
\end{equation}
\end{proposition}
\subsection{Index of Notation}\label{a4}
\begin{enumerate}
\item $C^n(G), C_0^n(G)$ are the subspaces of (normalized) smooth n-cochains on $G,$ \ref{normc}
\item $C^n(\mathfrak{g})$ are pointwise multilinear maps on $\mathfrak{g},$ \ref{multil}
\item $\Lambda^nG^*$ is the subspace of completely antisymmetric n-cochains, \ref{antic}
\item $S_0^nG^*$ is the subspace of normalized, completely antisymmetric n-cochains, \ref{well}
\item  $\mathcal{A}^n_0G^*=S^{n}_0G^*\oplus\Lambda^{n} G^*$ is the subspace of normalized $n$-cochains invariant under even permutations, \ref{even}
    \item $G_{\textup{loc}}$ is a local Lie groupoid, \ref{locg}
    \item $\textup{Pair}\,\textup{M}_{\textup{loc}}$ is the local pair groupoid, \ref{locp}, \ref{pair}
    \item $\delta^*, d$ are the Lie groupoid and Lie algebroid differentials, respectively, \ref{diffg}, \ref{diffa}
    \item $\mathcal{T}_M$ is the set of all smooth triangulaions of $M,$ $|\Delta_M|$ is a triangulation of $M,$ $\Delta_M$ is the associated simplicial complex/simplicial set, \ref{trian}
    \item $VE_{\bullet}$ is the graded van Est map, \ref{ve1}, \ref{ve2}
    \item $\mathcal{H}_i(TX)$ are pointwise homogeneous maps of degree $i,$ \ref{pointh}
    \item $\Pi_1(M)$ is the fundamental groupoid, \ref{fundg}

\end{enumerate}
\end{appendices}

\end{document}